\pdfoutput=1
\documentclass{article}
\usepackage{libertine}
\usepackage{xspace}

\providecommand\NOINDEX[1]{}
\providecommand\ASLSTYLE[1]{}
\usepackage[anglais, surreels]{perso}
\usepackage{fancybox}
\newcommand{\K}{\Kbb}
\newcommand{\R}{\Rbb}
\newcommand{\omegaunCK}{\omega_1^{\textrm{CK}}}

\usepackage{thm-restate}

\usepackage{titling}
\usepackage{stackengine}
\setstackgap{L}{.8\baselineskip}

\title{Surreal fields stable under exponential and  logarithmic functions}
\usepackage{authblk}
\author{\stackunder{Olivier Bournez}{\small olivier.bournez@lix.polytechnique.fr}}
\author{\stackunder{Quentin Guilmant}{\small quentin.guilmant@lix.polytechnique.fr}}
\affil{École Polytechnique, LIX, 91128 Palaiseau Cedex, France\newline
\small{This work was partially supported by ANR Project $\partial$IFFERENCE.}}
\date{}
\newtheorem{Blabla}{Blabla}[section]
\newtheorem{Theorem}[Blabla]{Theorem}
\newtheorem{Proposition}[Theorem]{Proposition}
\newtheorem{Lemma}[Theorem]{Lemma}
\newtheorem{Corollary}[Theorem]{Corollary}
\theoremstyle{definition}
\newtheorem{Definition}[Theorem]{Definition}

\newtheorem{Remark}[Theorem]{Remark}

\begin{document}
	\maketitle
\begin{abstract}
	Surreal numbers, that were initiated by Conway and popularized by Knuth, have a very rich and elegant theory, rooted in strong fundamental nice set theoretic foundations. This class of numbers, denoted by $\Nobf$, includes simultaneously the ordinal numbers and the real numbers, and forms a universal huge real closed field: It is universal in the sense that any real closed field can be embedded in it. Following Gonshor, surreal numbers can also be seen as signs sequences of ordinal length, with some exponential and logarithmic functions that extend the usual functions over the reals.  $\Nobf$   can actually also be seen as an elegant particular (generalized) power series field with real coefficients, namely Hahn series with exponents in $\Nobf$ itself. 
 	It can also be considered as a particular field of transseries, providing tools to do some analysis and asymptotic analysis for  functions over the continuum, providing natural concepts for discussing  hyperexponential or sublogarithm functions, and their asymptotics. 	
 
	In this article, we consider stability of subfields of $\Nobf$ under exponential and logarithmic functions. Namely, we consider $\Nolambda$ the set surreal numbers whose signs 
  	sequences have length less than $\lambda$ where $\lambda$ is some ordinal. Extending the discussion from van den Dries and Ehrlich,  we show that $\Nolambda$ is stable by exponential and logarithm \tiff $\lambda$ is some $\epsilon$-number. These authors have also proved that  $\Nolambda$  can be seen as a Hahn series with exponents in  $\Nolambda$  and length less than $\lambda$ \tiff $\lambda$ is some regular cardinal. 
  	Motivated in a longer term by computability issues using ordinal machines, we consider subfields stables by exponential and logarithmic functions defined by Hahn series that does not require to go up to cardinal lengths and exponents, i.e. where $\lambda$ can be some ordinal and not necessarily a cardinal. 
	
	We prove that $\Nobf$ can be expressed as a strict hierarchy of subfields stable by exponential and logarithmic functions. The definition of each level of this hierarchy is given in terms of Hahn series of length $\mu$, where $\mu$ lives in the multiplicative ordinals smaller than $\lambda$.  This provides many explicit examples of subfields of $\Nobf$ stable by exponential and logarithmic functions, and does not require to go up to a cardinal $\lambda$ to provide such examples.  
  \end{abstract}

\section{Introduction}
The class of surreal numbers has been introduced by Conway in \cite{conway2000numbers}, and then popularized by Knuth \cite{knuth1974surreal}, and then formalized later on by Gonshor \cite{gonshor1986introduction}, and by many other authors.  The general initial idea from Conway to define the class of surreal numbers, based on a concept of ``simplicity'',  is rooted on a unification of Dedekind's construction of real numbers in terms of cuts of the rational numbers, and of von Neumann's construction of ordinal numbers by transfinite induction in terms of set membership.

Following the alternative presentation from Gonshor in \cite{gonshor1986introduction}, a surreal number corresponds to an ordinal-length sequence over $\{+,-\}$, that we call a  \textbf{signs sequence}. Basically, the idea is that such sequences are ordered lexicographically, and have a tree-like structure. Namely, a $+$ (respectively $-$) added to a sequence $x$ denotes the simplest number greater (resp. smaller) than $x$ but smaller (resp. greater) than all the prefixes of $x$ which are greater (resp. smaller) than $x$.   With this definition of surreal numbers, it is possible to define operations such as addition, substraction, multiplication, division, and this yields a structure of real closed field.  

This class of numbers, denoted by $\Nobf$, is fascinating as it includes simultaneously the ordinal numbers  and the real numbers: Ordinal $\lambda$ is interpreted by the sequence of $+$ of length $\lambda$, and signs sequences of finite length correspond to the dyadic rational numbers, while sequences of length $\omega$ include real numbers as well as new numbers such $1/\omega$,  the inverse of ordinal $\omega$. 
Following Gonshor \cite{gonshor1986introduction}, based on ideas from Kruskal, it is also possible to define consistently classical functions such as the exponential function and the  logarithmic function over $\Nobf$, and to do analysis of this fields of numbers. 

A very elegant and nice property, rooted in very fundamental and accessible set theoretic foundations, is that this huge real closed field has a very strong \textbf{universality} property: It can be considered as ``the'' field that includes ``all numbers great and small'' \cite{ehrlich2012absolute}.  In particular, any divisible ordered Abelian group is isomorphic to an initial subgroup of $\Nobf$, and any real closed field is isomorphic to an initial subfield of $\Nobf$ \cite[Theorems 9 and 19]{ehrlich2001number}, \cite[Theorems 28 and 29]{conway2000numbers}, $\Nobf$ being itself a real closed field. 

Furthermore, the class of surreal numbers have been shown to unify many \textit{a priori} very different concepts and definitions, and with strong model theoretic properties, and in particular the results on $\Nobf$ described so far agree with important results on the model theory of the ordered field of real numbers with restricted analytic functions and the exponential function. See in particular \cite{DriesEhrlich01}, from which we will mention many results, but about which we decided here not to discuss model theoretic aspects. $\Nobf$ can also be equipped with a derivation, so that it can be considered as a fields of transseries \cite{berarducci2018surreal}. See example \cite{mantova2017surreal} for a survey of fascinating recent results in all these directions.  

From all these facts, we believe this is fundamental to better understand $\Nobf$ and its subfields,  as the exponential and logarithmic functions are among the simplest functions one would expect to start to do analysis. This article originated from some studies of subfields of $\Nobf$ on which it could be easy to do some computations, both in the sense of computer algebra, but also suitable to discuss computable models over the continuum in the spirit of results such as \cite{JournalACM2017} (considering computations over superreal fields) or in the spirit of \cite{galeotti2016candidate}, considering computations over the ordinals, generalizing classical computability. We believe that the collection of results presented in this article, are only first steps in these directions, but have their interest in their own. 

More concretely, $\Nobf$ can also be seen as a field of (generalized) power series with real coefficients, namely as Hahn series where exponents are surreal numbers themselves. More precisely, write $\HahnField{\K}{G}$ for the set of Hahn series with coefficients in $\K$ and terms corresponding to elements of $G$, where $\K$ is a field, and $G$ is some divisible ordered Abelian group: This means that $\HahnField{\K}{G}$ corresponds to formal power series of the form $s=\sum_{g \in S} a_{g} t^{g}$, where $S$ is a well-ordered subset of $G$ and $a_{g} \in \K .$ The support of $s$ is $\supp(s)=\enstq{g \in S}{a_{g} \neq 0}$ and the length of the serie of $s$ is the order type of $\supp(s)$. The field operations on $\HahnField{\K}{G}$ are defined as expected, considering elements of $\HahnField{\K}{G}$ as formal power series.
We have $\Nobf= \HahnField{\R}{\Nobf}$.


The purpose of this article is to study stability of subfields of $\Nobf$ by exponential and logarithm.  To do so, we will consider restrictions of $\Nobf$ by various subclasses of \textbf{ordinals}. 

In particular, it is natural to 
restrict the class of ordinals allowed in the ordinal sum, namely by restricting to ordinals up to some ordinal $\lambda$: Given some ordinal $\gamma$ (or more generally a class of ordinals), we write $\HahnFieldOrd{\K}{G}{\gamma}$ for the restriction of $\HahnField{K}{G}$ to formal power series whose support has an order type in $\gamma$ (that is to say, corresponds to some ordinal less than $\gamma$). We  have of course $\Nobf= \HahnFieldOrd{\R}{\Nobf}{\Ord}$.


In a similar spirit, it is natural to restrict signs sequences to some ordinal $\alpha$: We write $\Nolt\alpha$ for the class of surreal number whose signs 
  sequences have length less than $\alpha$ where $\alpha$ is some ordinal (or more generally some class of ordinals). 
  
Van den Dries and Ehrlich have proved the following: 

  \begin{Theorem}[\cite{DriesEhrlich01,van2001erratum}]
  \label{thm:structureNolambda}
   The ordinals $\lambda$
  such that $\Nolambda$ is closed under the various fields operations of $\No$ can be
  characterized as follows:
  \begin{itemize}
  \item $\Nolambda$ is an additive subgroup of $\No$ \tiff
    $\lambda=\omega^\alpha$ for some ordinal $\alpha$.
  \item $\Nolambda$ is a subring of $\No$ \tiff
    $\lambda=\omega^{\omega^\alpha}$ for some ordinal $\alpha$.
  \item $\Nolambda$ is a subfield of $\No$ \tiff
    $\omega^\lambda=\lambda$.
\end{itemize}
\end{Theorem}

Recall that one says that ordinal $\lambda$ is an \textbf{$\epsilon$-number} \tiff
$\omega^\lambda=\lambda$, where $\omega$ is  the first transfinite ordinal.

\begin{Theorem}[\cite{DriesEhrlich01,van2001erratum}]
Let $\lambda$ be any $\epsilon$-number. Then 
  $\Nolambda$ is a real closed field. 
\end{Theorem}

Extending these results from \cite{DriesEhrlich01,van2001erratum}, actually the following can be observed: 
\begin{restatable}{Theorem}{thNoltStableExpLn}
The following are equivalent:
	\label{thm:NoltStableExpLn} 
	\begin{itemize}
	\item $\Nolambda$ is a subfield of $\No$ stable by $\exp$, and $\ln$ 
	\item   $\Nolambda$ is a subfield of $\No$ 
	\item  $\lambda$ is some $\epsilon$-number.
	\end{itemize}
\end{restatable}

As we will often play with exponents of formal power series considered in the Hahn series, we propose to introduce the following notation:  We denote $$\SRF{\lambda}{\Gamma}=\HahnFieldOrd\Rbb{\Gamma}\lambda$$ when  $\lambda$ is an $\epsilon$-number and $\Gamma$ a divisible Abelian group.

As a consequence of MacLane's theorem (Theorem \ref{thm:macLane} below from \cite{maclane1939}, see also \cite[section 6.23]{alling1987foundations}), we know that $\SRF\lambda{\Nolt\mu}$ is a real-closed field when $\mu$ is a \textbf{multiplicative ordinal} (i.e. $\mu=\omega^{\omega^\alpha}$ for some ordinal $\alpha$) and $\lambda$ an $\epsilon$-number.

Furthermore:

\begin{restatable}[{\cite[Proposition 4.7]{DriesEhrlich01}}]{Theorem}{thEhrlichquatresept}
\label{th:Ehrlichquatresept}
Let $\lambda$ be an $\epsilon$-number. Then
\begin{enumerate}
\item The field $\Nolambda$ can be expressed as \begin{equation} \label{eq:troissix} 
\Nolambda=\bigcup_{\mu} \SRF{\lambda}{\Nolt\mu},
\end{equation}
where $\mu$ ranges over the additive ordinals less than $\lambda$ (equivalently, $\mu$ ranges over the multiplicative ordinals less than $\lambda$ ).
\item  $\Nolambda$ is a real closed subfield of $\No$, and is closed under the restricted analytic functions of $\No$.
\item $\Nolambda=\SRF{\lambda}{\Nolt\lambda}$ if and only if $\lambda$ is a regular cardinal.
\end{enumerate}
\end{restatable}


Actually, even if we always can write $\Nolambda$ as an increasing union of fields by Equation \eqref{eq:troissix},  and even if $\Nolt\lambda$ is stable under exponential and logarithmic functions (Theorem \ref{thm:NoltStableExpLn}) none of the fields in this union has stability property beyond the fact that they are fields. Indeed, we prove the following:

\begin{restatable}{Proposition}{firstinstable}
\label{prop:instable}
	$\SRF\lambda{\Nolt\mu}$ is never closed under exponential function for $\mu<\lambda$ a multiplicative ordinal. 
\end{restatable}

As we said, the purpose of this article is to study stability of subfields of $\Nobf$ by exponential and logarithm.

Some previous work actually studied $\SRF\kappa{\Nolt\kappa}$ for \textbf{regular cardinals $\kappa$}. More precisely, \cite[Proposition 4.1 and Corollary 5.5]{DriesEhrlich01} and above theorems show that $\Nolt\kappa = \SRF\kappa{\Nolt\kappa}$ \tiff $\kappa$ is a regular cardinal. They also showed that it is stable under exponential and logarithm under the same hypothesis. In 2016, Galeotti studied the completion of $\Nolt\kappa$ in \cite{galeotti2016candidate} to make it a candidate for the generalization of $\R$ giving up the Archimedean property. However, in the purpose of effectiveness and representations for ordinal Turing machines, we want to consider \textbf{ordinals} as small as possible to identify natural subfields stable by exponential and logarithm. 

Notice that, using classical techniques from the theory of admissible sets \cite{barwise1975admissible}, it can rather easily be established that we have stability by exponential and logarithms for surreal fields in some admissible sets. More precisely, reproducing the proofs of stability in Kripke-Platek set theory proof, Theorem \ref{th:Ehrlichquatresept} holds in an admissible set $A$ (strictly greater than $\omega$).  In this setup,  for example, the first non-computable ordinal $\omegaunCK$ (which is also the first admissible ordinal greater than $\omega$),  is for example actually a regular cardinal \textit{with respect to $\Sigma_{1}$ definable functions}, that is to say  a regular cardinal \textit{with respects to functions in $L_{\omegaunCK}$}.

We mean, we have the following:

\begin{Theorem}
	Theorem \ref{th:Ehrlichquatresept} holds in any admissible set $A$: If  $\Nolambda^{A}$ denotes the  surreal numbers  \textit{in $A$} of length less than $\lambda$ for $\lambda$ an $\epsilon$-number, and $\HahnFieldOrd{\Rbb}{\Gamma}\lambda^{A}$ is  the field of  Hahn series \textit{in $A$} of $\HahnFieldOrd{\Rbb}{\Gamma}\lambda$
	we have that $\Nolt{\alpha}^{A}$ and $\HahnFieldOrd\Rbb{\Nolt\alpha^{A}}\alpha^{A}$ are isomorphic and stable under exponential and logarithmic functions, where $\alpha$ is the ordinal of $A$ (the first ordinal not in $A$).
	
	In particular, this holds for $\alpha=\omegaunCK$ and $A=L_{\omegaunCK}$.
\end{Theorem}

Its proofs consists basically in reproducing the arguments of  Theorem \ref{th:Ehrlichquatresept}  Kripke-Platek set theory, and observing that same conclusions hold.

With the previous kind of results, and techniques, it is then possible to build subfields of $\Nobf$ that are stable under exponential and logarithm and that involve countable ordinals. However, this sounds as cheating since we just change what is authorized in the signs sequences and in the series so that the theorem holds. Moreover, it does not change the fact that we handle non-computable ordinals. In this article, our ambition is to do better: We are motivated by considering ordinals that could be really smaller than $\omegaunCK$ (and \textit{a fortiori} regular cardinals) hence that are computable.

To achieve that purpose, we will have to handle carefully the $\epsilon$-numbers that are involved. Recall that there is some enumeration $(\epsilon_{\alpha})_{\alpha \in \Ord}$of $\epsilon$-numbers: 
Any $\epsilon$-number ordinal $\lambda$ is $\epsilon_\alpha$ for some ordinal $\alpha$.

\begin{Definition}[Canonical sequence defining an $\epsilon$-number]
	\label{def:gammaLambda}
	Let $\lambda$ be an $\epsilon$-number.  Ordinal $\lambda$ can always  be written as $\lambda=\sup\suitelt{e_\beta}\beta{\gamma_\lambda}$ for some \textbf{canonical sequence}, where $\gamma_\lambda$ is the length of  this sequence, and this sequence is defined as follows:
	\begin{itemize}
		\item If $\lambda=\epsilon_{0}$ then we can write $\epsilon_0=\sup\{\omega,\omega^\omega,\omega^{\omega^\omega},\dots\}$ and  we take\linebreak $\omega,\omega^\omega,\omega^{\omega^\omega},\dots$ as canonical sequence for $\epsilon_0$. Its length is $\omega$, and for $\beta<\lambda$, $e_\beta$ is $\omega^{\iddots^\omega}$ where there are $\beta$ occurrences of $\omega$ in the exponent.
		
		\item If $\lambda=\epsilon_{\alpha}$, where $\alpha$ is a non-zero limit ordinal, then we can write $\lambda=\underset{\beta<\alpha}\sup\epsilon_\beta$ and we take $\suitelt{\epsilon_\beta}\beta\alpha$ as the canonical sequence of $\lambda$. Its length is $\alpha$ and for $\beta<\alpha$, $e_\beta=\epsilon_\beta$.
		
		\item If $\lambda=\epsilon_{\alpha}$, where  $\alpha$ is a successor ordinal, then we can write 
		$$\lambda=\sup\{\epsilon_{\alpha-1}, {\epsilon_{\alpha-1}}^{\epsilon_{\alpha-1}}, {\epsilon_{\alpha-1}}^{{\epsilon_{\alpha-1}}^{\epsilon_{\alpha-1}}},\dots\}$$
		and we take $\epsilon_{\alpha-1}, {\epsilon_{\alpha-1}}^{\epsilon_{\alpha-1}}, {\epsilon_{\alpha-1}}^{{\epsilon_{\alpha-1}}^{\epsilon_{\alpha-1}}},\dots$ as the canonical sequence of $\lambda$. Its length is $\omega$, and for $\beta<\omega$, $e_\beta={\epsilon_{\alpha-1}}^{\iddots^{\epsilon_{\alpha-1}}}$ where there are $\beta$ occurrences of $\epsilon_{\alpha-1}$ in the exponent.
	\end{itemize}
\end{Definition}

For example, the canonical sequence of $\epsilon_1$ is $\epsilon_0,{\epsilon_0}^{\epsilon_0},{\epsilon_0}^{{\epsilon_0}^{\epsilon_0}},\dots$, the canonical sequence of $\epsilon_\omega$ is $\epsilon_0,\epsilon_1,\epsilon_2,\dots$, the canonical sequence of $\epsilon_{\omega2}$ is \linebreak $\epsilon_0,\epsilon_1,\epsilon_2,\dots,\epsilon_\omega,\epsilon_{\omega+1},\dots$ and the canonical sequence of $\epsilon_{\omega2+1}$ is \linebreak ${\epsilon_{\omega2}},{\epsilon_{\omega2}}^{\epsilon_{\omega2}}, {\epsilon_{\omega2}}^{{\epsilon_{\omega2}}^{\epsilon_{\omega2}}},\dots$

\begin{Definition}
	\label{def:uparrow}
	Let $\Gamma$ be an Abelian subgroup of \Nobf{} and $\lambda$ be an $\epsilon$-number whose canonical sequence is $\suitelt{e_\beta}\beta{\gamma_\lambda}$. 
	We denote $\Gamma^{\uparrow\lambda}$ for the family of group $\suitelt{\Gamma_\beta}\beta{\gamma_\lambda}$ defined as follows:
	\begin{itemize}
		\item $\Gamma_0=\Gamma$;
		
		\item $\Gamma_{\beta+1}$ is the group generated by $\Gamma_\beta$, $\SRF{e_\beta}{g\pa{(\Gamma_\beta)^*_+}}$ and $\enstq{h(a_i)}{\aSurreal\in\Gamma_\beta}$
		where $g$ and $h$ are Gonshor's functions associated to exponential and logarithm (see Section \ref{sec:expoLn} below for some details);
		
		\item For limit ordinal numbers $\beta$, $\Gamma_\beta=\Unionlt\gamma\beta \Gamma_\gamma$.
	\end{itemize}
\end{Definition}

When considering a family of set $(S_i)_{i\in I}$, we denote
$$
	\SRF\lambda{(S_i)_{i\in I}} = \Unionin iI \SRF\lambda{S_i}
$$
In particular,
$$
	\SRF\lambda{\Gamma^{\uparrow\lambda}} = \Unionlt i{\gamma_\lambda}\SRF\lambda{\Gamma_i}
$$
\begin{Remark}\label{rk:inclusionUparrow}
	By construction, if $\Gamma\subseteq\Gamma'$ then $\SRF\lambda{\Gamma^{\uparrow\lambda}} \subseteq \SRF\lambda{{\Gamma'}^{\uparrow\lambda}}$.
\end{Remark}
The idea behind the definition of $\Gamma^{\uparrow}$ is that at step $i+1$ we add new elements to close $\SRF\lambda{\Gamma_i}$ under exponential and logarithm. The reason why we add $\SRF{e_\beta}{g\pa{(\Gamma_\beta)^*_+}}$ to $\Gamma_\beta$ rather than $\SRF\lambda{g\pa{(\Gamma_\beta)^*_+}}$ is that we want to keep control on what we add in the new group. It will be actually useful in the proofs. 

That said, we prove the following theorem:

\begin{restatable}{Theorem}{thmSRFGammaUpStableExpLn}
	\label{thm:SRFGammaUpStableExpLn}
	Let $\Gamma$ be an Abelian subgroup of \Nobf{} and $\lambda$ be an $\epsilon$-number, then
	$\SRF\lambda{\Gamma^{\uparrow\lambda}}$ is stable under exponential and logarithmic functions.
\end{restatable}

With such a notion, we can now make a link between the two types of field involved in Theorems \ref{thm:NoltStableExpLn} and \ref{thm:SRFGammaUpStableExpLn}. More precisely, the fields $\SRF\lambda{\Gamma^{\uparrow\lambda}}$ are part of the fields $\Nolt\lambda$.

\begin{restatable}{Theorem}{thmNolambdaDecompCorpsStables}
	\label{thm:NolambdaDecompCorpsStables}
	$\Nolambda=\bigcup_{\mu} \SRF{\lambda}{{\Nolt{\mu}}^{\uparrow \lambda}}$, where $\mu$ ranges over the additive ordinals less than $\lambda$ (equivalently, $\mu$ ranges over the multiplicative ordinals less $\lambda$),
\end{restatable}
Notice that now,  $\Nolambda$ is  expressed as a increasing union of fields, each of them closed by $\exp$ and $\ln$. Indeed, by definition, if $\mu<\mu'$ then $\Nolt\mu\subseteq\Nolt{\mu'}$ and Remark \ref{rk:inclusionUparrow} gives $\SRF{\lambda}{{\Nolt{\mu}}^{\uparrow \lambda}} \subseteq \SRF{\lambda}{{\Nolt{\mu'}}^{\uparrow \lambda}}$. 

We finally state that each field $\SRF{\lambda}{{\Nolt{\mu}}^{\uparrow \lambda}}$ is interesting for itself since none of them is $\Nolambda$. More precisely: 

\begin{restatable}{Theorem}{thmhierarchieUparrow}
\label{thm:hierarchieUparrow}
	For all $\epsilon$-number $\lambda$, the hierarchy in previous theorem is strict:
	$$\SRF{\lambda}{{\Nolt{\mu}}^{\uparrow \lambda}} \subsetneq \SRF{\lambda}{{\Nolt{\mu'}}^{\uparrow \lambda}}$$	
	for all multiplicative ordinals $\mu$ and $\mu'$ such that $\omega<\mu<\mu'<\lambda$.
\end{restatable}

This article is organized as follows. Section \ref{sec:introsurreals} recalls basics of the concepts and definitions of the theory of surreal numbers, and fixes the notations used in the rest of the paper. Section \ref{sec:erhlichandco} recalls what is known about the stability properties of various subfields of $\Nobf$ according to their signs sequence representation or Hahn series representation.  In Section \ref{sec:expoLn} we recall the definitions of exponential and logarithm, and based on existing literature, we establish some various bounds and statements on corresponding functions $g$ and $h$, needed for the rest of the article. In particular, this provides a  proof of Theorem \ref{thm:NoltStableExpLn} and of  Proposition \ref{prop:instable} that we give at the end of Section \ref{sec:expoLn}. 
Section \ref{sec:proofs} is devoted to prove Theorems \ref{thm:SRFGammaUpStableExpLn} and \ref{thm:NolambdaDecompCorpsStables} and the strictness of the hierarchy (Theorem \ref{thm:hierarchieUparrow}).

\section{Surreal numbers}
\label{sec:introsurreals}


We assume some familiarity with the ordered field of surreal numbers (refer to  \cite{conway2000numbers,gonshor1986introduction} for presentations) which we denote by $\No$. In this section we give a brief presentation of the basic definitions and results, and we fix the notations that will be used in the rest of the paper.

\subsection{Order and simplicity}

The class $\No$ of surreal numbers can be defined either by transfinite recursion, as in  \cite{conway2000numbers} or by transfinite length sequences of $+$ and $-$ as done in 
\cite{gonshor1986introduction}. We  will mostly follow \cite{gonshor1986introduction}, as well as \cite{berarducci2018surreal} for their presentation.

We introduce the class $\No = 2^{<\On}$ of all binary sequences of some ordinal length $\alpha \in \On$, where $\On$ denotes the class of the ordinals. In other words,  $\No$ corresponds to functions of the form $x : \alpha \to \{-,+\}$. The \textbf{length} (sometimes also called \textbf{birthday} in  literature) of a surreal number $x$ is the ordinal number $\alpha = \dom(x)$. We will also write $\alpha=\length{x}$ (the point of this notation is to ``count'' the number of pluses and minuses).
Note that $\No$ is not a set but a proper class, and all the relations and functions we shall define on $\No$ are going to be class-relations and class-functions, usually constructed by transfinite induction.

We say that $x \in \No$ is \textbf{simpler} than $y \in \No$, denoted $x \simpler y$, i.e., if $x$ is a strict \textbf{initial segment} (also called \textbf{prefix}) of $y$ as a binary sequence. We say that $x$ is simpler than or equal to $y$, written $x \simplereq y$, if $x \simpler y$ or $x = y$ i.e., $x$ is an initial segment of $y$. The simplicity relation is a binary tree-like partial order on $\No$, with the immediate successors of a node $x\in\No$ being the sequences $x_-$ and $x_+$ obtained by appending $-$ or $+$ at the end of the signs sequence of $x$. Observe in particular that the simplicity relation $\simpler $ is well-founded, and the empty sequence, which will play the role of the number zero, is simpler than any other surreal number. 

We can introduce a total order $<$ on $\No$ which is basically the lexicographic order over the corresponding sequences: More precisely,  we consider the order  $-<\square<+$ where $\square$ is the blank symbol. Now to compare two signs sequences,  append blank symbols to the shortest so that they have the same length. Then,  just compare them with the corresponding lexicographic order to get the total order $<$.

Given two sets $A \subseteq \No$ and $B \subseteq \No$ with $A < B$ (meaning that $a < b$ for all $a \in A$ and $b \in B$),
it is quite easy to understand why there is a simplest surreal number, denoted $\crotq AB$ such that $A<\crotq AB < B$.  However, a formal proof is long. See \cite[Theorem 2.1]{gonshor1986introduction} for details.
If $x=\crotq AB$, we say that Such a pair $\crotq{A}{B}$ is \textbf{representation} of $x$.

Every surreal number $x$ has several different representations $x = \crotq{A}{ B} = \crotq{A'}{ B'}$, for instance, if $A$ is cofinal with $A'$ and $B$ is coinitial with $B'$. In this situation, we shall say that $\crotq{A}{ B} = \crotq{A'}{ B'}$ by cofinality. On the other hand, as discussed in \cite{berarducci2018surreal}, 
it may well happen that $\crotq{A}{ B} = \crotq{A'}{ B'}$ even if $A$ is not cofinal with $A'$ or $B$ is not coinitial with $B'$. The \textbf{canonical representation} $x = \crotq{A}{ B}$ is the unique one such that $A \cup B$ is exactly the set of all surreal numbers strictly simpler than $x$. Indeed it turns out that is $A=\enstq{y\sqsubset x}{y<x}$ and $B=\enstq{y\sqsubset x}{y>x}$, then $x=\crotq AB$.

\begin{Remark} 
By definition, if $x = \crotq{A}{ B}$ and $A < y < B$, then $x \simplereq y$.
\end{Remark}
To make the reading easier we may forget $\{\}$ when writing explicitly $A$ and $B$. For instance $\crotq xy$ will often stand for $\crotq{\{x\}}{\{y\}}$ when $x,y\in\No$.


\subsection{Field operations}


Ring operations $+$, $·$ on $\No$ are defined by transfinite induction on simplicity as follows:
$$x+y:=\crotq{x' +y, x+y'}{x'' +y, x+y''}$$
$$
xy := \crotq{x'y + xy' - x'y', x''y + xy'' - x''y''} {x'y + xy'' - x'y'', x''y + xy' - x''y'}
$$
where $x'$ (resp. $y'$) ranges over the numbers simpler than $x$ (resp. $y$) such that $x' < x$ (resp. $y'<y$) and $x''$ (resp. $y''$) ranges over the numbers simpler than $x$ (resp. $y$) such that $x < x''$ (resp. $y<y''$); in other words, when $x = \crotq{x'}{x''}$ and $y = \crotq{y'} {y''}$ are the canonical representations of $x$ and $y$ respectively. The expression for the product may seem not intuitive, but actually, it is basically inspired by the fact that we expect $(x-x')(y-y')>0$, $(x-x'')(y-y'')>0$, $(x-x')(y-y'')<0$ and $(x-x'')(y-y')<0$.

\begin{Remark}
The definitions of sum and product are uniform in the sense of \cite[page 15]{gonshor1986introduction}. Namely the equations that define $x + y$ and $xy$ does not require the canonical representations of $x$ and $y$ but any representation. In particular, if $x=\crotq AB$ and $y=\crotq CD$, the variables $x', x'', y', y''$ may range over $A$, $B$, $C$, $D$ respectively.
\end{Remark}

It is an early result that these operations, together with the order, give $\No$ a structure of ordered field, and even a structure of real closed field (see \cite[Theorem 5.10]{gonshor1986introduction}). 
Consequently, there is a unique embedding of the rational numbers in $\No$ so we can
identify $\Qbb$ with a subfield of $\No$. 
Actually, the subgroup of the dyadic rationals $m/2^{n}\in \Qbb$, 
with $m\in\Zbb$ and $n \in \Nbb$, correspond exactly to the surreal numbers $s : k \to \{-, +\}$ of finite length $k \in \Nbb.$

The field $\R$ can be isomorphically identified with a subfield of $\No$ by sending $x\in\R$ to the number $\crotq{A}{ B}$ where $A\subseteq\No$ is the set of rationals (equivalently: dyadics) lower than $x$ and $B\subseteq\No$ is the set of (equivalently: dyadics) greater than $x$. This embedding is consistent with the one of $\Qbb$ into $\No$. We may thus write $\Qbb\subseteq\R\subseteq\No$. By \cite[page 33]{gonshor1986introduction}, the length of a real number is at most $\omega$ (the least infinite ordinal). There are however surreal numbers of length $\omega$ which are not real numbers, such as $\omega$ itself or its inverse that is a positive infinitesimal.  

The ordinal numbers can be identified with a subclass of $\No$ by sending the ordinal $\alpha$ to the sequence $s : \alpha \rightarrow\{+,-\}$ with constant value $+$. Under this identification, the ring operations of $\No$, when restricted to the ordinals $\Ord \subseteq \No$, coincide with the Hessenberg sum and product (also called natural operations) of ordinal numbers. Similarly, the sequence $s : \alpha\rightarrow\{+,-\}$ with constant value $-$ corresponds to the opposite (inverse for the additive law) of the ordinal $\alpha$, namely $-\alpha$. We remark that $x \in \Ord$ if and only if $x$ admits a representation of the form $x = \crotq AB$ with $B=\emptyset$, and similarly $x \in -\Ord$ if and only if we can write $x = \crotq AB$ with $A=\emptyset$.

Under the above identification of $\Qbb$ as a subfield of $\No$, the natural numbers $\Nbb \subseteq\Qbb$ are exactly the finite ordinals.

\subsection{Hahn series}
\label{sec:hahn}

%
%
%
%
\subsubsection{Generalities}
Let $\K$ be a field, and let $G$ be a divisible ordered Abelian group.

\begin{Definition}[Hahn series \cite{hahn1995nichtarchimedischen}]
The Hahn series (obtained from $\K$ and $G$) are formal power series of the form $s=\sum_{g \in S} a_{g} t^{g}$, where $S$ is a well-ordered subset of $G$ and $a_{g} \in \K .$ The support of s is $\supp(s)=\enstq{g \in S}{ a_{g} \neq 0}$ and the length of $s$ is the order type of $\supp(s)$.

We write $\HahnField{\K}{G}$ for the set of Hahn series with coefficients in $K$ and terms corresponding to elements of $G$. 
\end{Definition}

\begin{Definition}[Operations on $\HahnField{\K}{G}$]
The operations on $\HahnField{K}{G}$ are defined in the natural way:
 Let $s=\sum_{g \in S} a_{g} t^{g}, s^{\prime}=\sum_{g \in S^{\prime}} a_{g}^{\prime} t^{g}$, where $S, S^{\prime}$ are well
ordered.
\begin{itemize}
\item $s+s^{\prime}=\sum_{g \in S \cup S^{\prime}}\left(a_{g}+a_{g}^{\prime}\right) t^{g}$, where $a_{g}=0$ if $g \notin S$, and $a_{g}^{\prime}=0$ if $g \notin S^{\prime}$.
\item $s \cdot s^{\prime}=\sum_{g \in T} b_{g} t^{g}$, where $T=\enstq{g_{1}+g_{2}}{g_{1} \in S \wedge g_{2} \in S^{\prime}}$, and for
each $g \in T$, we set
$b_{g}=\Sum{g_{1}\in S, g_{2}\in S' | g_1+g_2=g}{} b_{g_{1}} \cdot b_{g_{2}}$
\end{itemize}
\end{Definition}

Hahn fields inherits a lot of from the structure of the coefficient field. In particular if $\K$ is algebraically closed, and if  $G$ is some divisible (i.e. for any $n\in\Nbb$ and $g\in G$ there is some $g'\in G$ such that $ng'=g$) ordered Abelian group, then the corresponding Hahn field is also algebraically closed. More precisely: 

\begin{Theorem}[Generalized Newton-Puiseux Theorem, Maclane \cite{maclane1939}]
	\label{thm:macLane}
	Let $G$ be a divisible ordered Abelian group, and let $\K$ be a field that is algebraically closed of characteristic $0$. Then $\HahnField{\K}{G}$ is also algebraically closed.
\end{Theorem}
As noticed in \cite{alling1987foundations}, we can deduce the following:
\begin{Corollary}
	Let $G$ be a divisible ordered Abelian group, and let $\K$ be a field that is real closed of characteristic $0$. Then $\HahnField{\K}{G}$ is also real closed.
\end{Corollary}

\begin{proof}
	$\K$ is real closed. That is to say that $-1$ is not a square in $\K$ and that $\K[i]$ is algebraically closed. Notice that $\K[i]((G))=\left(\K((G))\right)[i]$. Therefore, Theorem \ref{thm:macLane} ensures that $\left(\K((G))\right)[i]$ is algebraically closed. Also, $-1$ is not a square in $\K((G))$. Therefore, $\K((G))$ is real closed. 
\end{proof}


\subsubsection{Restricting length of ordinals}


In this article, will often restrict the class of ordinals allowed in the ordinal sum, namely by restricting to ordinals up to some ordinal $\lambda$. We then give the following notation:
\begin{Definition}[$\HahnFieldOrd{\K}{G}{\gamma}$]
Let $\lambda$ be some ordinal.  We define $\HahnFieldOrd{\K}{G}{\gamma}$ for the restriction of $\HahnField{K}{G}$ to formal power series whose support has an order type in~$\gamma$ (that is to say, corresponds to some ordinal less than $\gamma$).
\end{Definition}


%

\begin{Theorem}
Assume $\gamma$ is some $\epsilon$-number. Then $\HahnFieldOrd{\K}{G}{\gamma}$ is a field.
\end{Theorem}

\begin{proof}
This basically relies on the observation that the length of the inverse of some Hahn series in this field remains in the field: This is basically a consequence of Proposition \ref{prop:orderTypeMonoid} below:
\end{proof}

\begin{Proposition}[{\cite[Corollary 1]{weiermannMaximalOrderType}}]
	\label{prop:orderTypeMonoid}
	Let $\Gamma$ be an ordered abelian group and $S\subseteq\Gamma_+$ be a well-ordered subset with order type $\alpha$. Then, $\inner S$, the monoid generated by $S$ in $\Gamma$ is itself well-ordered with order type at most $\omega^{\hat{\alpha}}$
	where, if $\alpha$ is in Cantor normal form
	$$\alpha=\Sum{i=1}{n}\omega^{\alpha_i}n_i$$
	then
	$$\hat\alpha = \Sum{i=1}{n}\omega^{\alpha_i'}n_i$$
	where $\beta'=\begin{accolade}
			\beta+1 & \text{if $\beta$ is an $\epsilon$-number}\\
			\beta & \text{otherwise}
		\end{accolade}$.
	In particular, $\inner S$ has order type at most $\omega^{\omega\alpha}$ (commutative multiplication over ordinals).
\end{Proposition}

We also get:

\begin{Proposition}[{\cite[Lemma 4.6]{DriesEhrlich01}}] 
	\label{prop:hahnFieldRealClosed}
Assume $\K$ is some real closed field, and $G$ is some abelian divisible group.  Then $\HahnFieldOrd{\K}{G}{\gamma}$ is real closed.
\end{Proposition}

Actually, this was stated in \cite[Lemma 4.6]{DriesEhrlich01} for the case $\K=\R$, but the proof ony uses the fact that $\R$ is real-closed. 

\subsubsection{Normal form theorem for surreal numbers}

\begin{Definition}
	For $a$ and $b$ two surreal numbers, we define the following relations: 
	\begin{itemize}
		\item $a\prec b$ if for all $n\in\Nbb$, $n|a|<|b|$.
		\item $a\preceq b$ if there is some natural number $n\in\Nbb$ such that $|a|<n|b|$.
		\item $a\asymp b$ if $a\preceq b$ and $b\preceq a$.
	\end{itemize}
\end{Definition}
With this definition, $\preceq$ is a preorder and $\prec$ is the corresponding strict preorder. The associated equivalence relation is $\asymp$ and the equivalence classes are the Archimedean classes.

\begin{Theorem}[{\cite[Theorem 5.1]{gonshor1986introduction}}]
	For all surreal number $a$ there is a unique positive surreal $x$ of minimal length such that $a\asymp x$.
\end{Theorem}
The unique element of minimal length in its Archimedean class has many properties  similar to those of exponentiation:

\begin{Definition}
	For all surreal number $a$ written in canonical representation $a=\crotq{a'}{a''}$, we define
	$$
		\omega^a=\crotq{0,\enstq{n\omega^{a'}}{n\in\Nbb}}{\enstq{\frac1{2^n}\omega^{a''}}{n\in\Nbb}}
	$$
	we call such surreal numbers \textbf{monomials}.
\end{Definition}
Actually this definition is uniform (\cite[Corollary 5]{gonshor1986introduction}) and therefore, we can use any representation of $a$ in this definition. Another point is that we can easily check that this notation is consistent with the ordinal exponentiation. More precisely, if $a$ is an ordinal, $\omega^a$ is indeed the ordinal corresponding to the ordinal exponentiation (see \cite[Theorem 5.4]{gonshor1986introduction}). Finally, as announced, this definition gives the simplest elements among the Archimedean classes.

\begin{Theorem}[{\cite[Theorem 5.3]{gonshor1986introduction}}]
	A surreal number is of the form $\omega^a$ if and only if it is simplest positive element in its Archimedean class. More precisely,
	$$
		\forall a\in\No\qquad (\exists c\in\No\quad a=\omega^c)\implies (\forall b\in\Nobf\quad b\asymp a\implies a\sqsubseteq |b|)
	$$
\end{Theorem}

Elements of the form $\omega^a$ are by definition positive and have the following property: 

\begin{Proposition}[{\cite[Theorem 5.4]{gonshor1986introduction}}]
	We have 
	\begin{itemize}
		\item $\omega^0=1$
		\item $\forall a,b\in\No\qquad \omega^a\omega^b=\omega^{a+b}$
	\end{itemize}
\end{Proposition}

Thanks to this definition of the $\omega$-exponentiation, we are now ready to expose a normal form for surreal numbers which is analogous to the Cantor normal form for ordinal normal. 

\begin{Definition}[{\cite[Section 5C, page 59]{gonshor1986introduction}}]
	For $\nu$ an ordinal number, $\suitelt{r_i}i\nu$ a sequence of non-zero real numbers and $\suitelt{a_i}i\nu$ a decreasing sequence of surreal numbers, we define $\aSurreal$ inductively as follows:
	\begin{itemize}
		\item If $\nu=0$, then $\aSurreal=0$
		\item If $\nu=\nu'+1$ then $\aSurreal=\aSurrealPrefix+r_{\nu'}\omega^{a_{\nu'}}$
		\item If $\nu$ is a limit ordinal,
		$$
			\aSurreal=\crotq{\enstq{\aSurrealPrefix + s\omega^{a_{\nu'}}}{\begin{array}{c}
						\nu'<\nu\\ s<r_{\nu'}
					\end{array}}}{\enstq{\aSurrealPrefix+s\omega^{a_{\nu'}}}{\begin{array}{c}
					\nu'<\nu\\ s>r_{\nu'}
				\end{array}}}
		$$
	\end{itemize}
\end{Definition}
Note that if $0$ is seen as a limit ordinal, then both definition are consistent.

\begin{Theorem}[{\cite[Theorem 5.6]{gonshor1986introduction}}]
	\label{thm:normalForm}
	Every surreal number can be uniquely written as $\aSurreal$. This expression will be called its normal form. 
\end{Theorem}
Note that  if $a$ is an ordinal number, then its normal form coincides with its Cantor normal form. In such a sum, elements $r_i\omega^{a_i}$ will be called the \textbf{terms} of the serie.

\begin{Definition}
	A surreal number $a$ in normal form $a=\aSurreal$ is
	\begin{itemize}
		\item \textbf{purely infinite} if for all $i<\nu$, $a_i>0$.
		\item \textbf{infinitesimal} if for all $i<\nu$, $a_i<0$ (or equivalently if $a\prec 1$). 
		\item \textbf{appreciable} if for all $i<\nu$, $a_i\leq0$ (or equivalently if $a\preceq 1$).
	\end{itemize}
	If $\nu'\leq\nu$ is the first ordinal such that $a_i\leq0$, then $\aSurrealPrefix$ is called the \textbf{purely infinite part} of $a$. Similarly, if $\nu'\leq\nu$ is the first ordinal such that $a_i<0$, $\Sumlt {\nu'\leq i}\nu r_i\omega^{a_i}$ is called the \textbf{infinitesimal part} of $a$.
\end{Definition}

\begin{Theorem}[{\cite[Theorems 5.7 and 5.8]{gonshor1986introduction}}]
	\label{thm:normalFormOp}
	Operation over surreal numbers coincides with formal addition and formal multiplication over the normal forms. More precisely,
	
	$$
		\aSurreal+\Sumlt i{\nu'}s_i\omega^{b_i} = \Sum{x\in\No}{}t_x\omega^x
	$$
	where
	\begin{itemize}
		\item $t_x=r_i$ if $i$ is such that $a_i=x$ and there is no $i$ such that $b_i=x$.
		\item $t_x=s_i$ if $i$ is such that $b_i=x$ and there is no $i$ such that $a_i=x$.
		\item $t_x=r_i+s_j$ if $i$ is such that $a_i=x$ and $j$ is such that $b_j=x$
	\end{itemize}
	and
	$$
		\pa{\aSurreal}\pa{\Sumlt i{\nu'}s_i\omega^{b_i}}= \Sum{x\in\No}{}\pa{\Sum{\tiny\enstq{\begin{array}{c}
						i<\nu\\ j<\nu'
				\end{array}}{a_i+b_j=x}}{}r_is_j}\omega^x
	$$
\end{Theorem}

We stated that every surreal number has a normal form. However, in the other direction, it is possible to get back the sign expansion from a normal form. We first introduce a new concept that will be critical to express the sign expansion.

\begin{Definition}[Reduced sign expansion, Gonshor, \cite{gonshor1986introduction}]
	\label{def:reducedSignExpansion}
	Let $x=\Sum{i<\nu}{}r_i\omega^{a_i}$ be a surreal number. The reduced 
	sign expansion of $a_i$, denoted $a_i^\circ$ is inductively defined as follows:
	\begin{itemize}
		\item $a_0^\circ=a_0$
		\item For $i>0$, if $a_i(\delta)=-$ and if there is there is $j<i$ 
		such that for $\gamma\leq\delta$, $a_j(\gamma)=a_i(\gamma)$, then we discard
		the minus in position $\delta$ in the sign expansion of $a_i$.
		\item If $i>0$ is a non-limit ordinal and $(a_{i-1})_-$ (as a sign expansion)
		is a prefix of $a_i$, then we discard this minus after $a_{i-1}$ if $r_{i-1}$
		is not a dyadic rational number. 
	\end{itemize}
\end{Definition}
More informally, $a_i^\circ$ is the sign expansion obtained when copying $a_i$ omitting 
the minuses that have already been treated before, in an other exponent of the serie. 
We just keep the new one brought by $a_i$. However, the later case give a condition where
even a new minus can be omitted.

\begin{Theorem}[\cite{gonshor1986introduction}, Theorems 5.11 and 5.12]
	\label{thm:serieToSignExp}
	For $\alpha$ an ordinal and a surreal $a$, we write $|a[:\alpha]|_+$ for the (ordinal) number of pluses in $\alpha[:\alpha]$ the prefix of length of $\alpha$ of $x$. Then,
	\begin{itemize}
		\item The sign expansion of $\omega^a$ is as follows: we start with a plus and the for any ordinal $\alpha<|a|$ we add $\omega^{|a[:\alpha]|_++1}$ occurrences of $a(\alpha)$ (the sign in position $\alpha$ in the signs sequence of $a$).
		\item The sign expansion of $\omega^an$ is the signs sequence of $\omega^a$ followed by $\omega^{|a|_+}(n-1)$ pluses.
		\item The sign expansion of $\omega^a\f1{2^n}$ is the sign expansion of $\omega^a$ followed by $\omega^{|a|_+}n$ minuses.
		\item The sign expansion of $\omega^ar$ for $r$ a positive real is the sign expansion of $\omega^a$ to which we add each sign of $r$ $\omega^{|a|_+}$ times excepted the first plus which is omitted. 
		\item The sign expansion of $\omega^ar$ for $r$ a negative real is the sign expansion of $\omega^a(-r)$ in which we change every plus in a minus and conversely.
		\item The sign expansion of $\Sum{i<\nu}{}r_i\omega^{a_i}$ is the juxtaposition of the sign expansions of the $\omega^{a_i^\circ}r_i$
	\end{itemize}
\end{Theorem}

As a final note of this subsection, we give some bounds on the length of monomials and terms.

\begin{Lemma}[{\cite[Lemma 4.1]{DriesEhrlich01}}]
	\label{lem:lengthOmegaA}
	For all surreal number $a\in\Nobf$, $$\length a\leq \length{\omega^a}\leq\omega^{\length a}$$
\end{Lemma}

\begin{Lemma}[{\cite[Lemma 6.3]{gonshor1986introduction}}]
	\label{lem:lengthTerm}
	Let $x=\aSurreal$ a surreal number. We have for all $i<\nu$, $\length{r_i\omega^{a_i}}\leq\length x$.
\end{Lemma}

\subsubsection{Hahn series and surreal numbers}
As a consequence of Theorems \ref{thm:normalForm} and \ref{thm:normalFormOp}, the field $\No$ in in fact a Hahn serie field. More precisely,

\begin{Corollary}
	The fields $\No$ and $\R((t^{\No}))$ are isomorphic.
\end{Corollary}

\begin{proof}
	Sending $t^a$ to $\omega^{-a}$ for all surreal number $a$, we notice that all the definitions match to each other.
\end{proof}

Notice that we have of course $\Nobf= \HahnFieldOrd{\R}{\Nobf}{\Ord}$.

\section{Surreal subfields}
\label{sec:erhlichandco}

\subsection{Subfields defined by Gonshor's representation}


%
  Let $\Nolambda$ denote the set surreal number whose signs 
  sequences have length less than $\lambda$ where $\lambda$ is some ordinal.
  We have of course $\Nobf = \bigcup_{\lambda \in \On} \Nolambda$.  
  
Van den Dries and Ehrlich have proved the following: 

  \begin{Theorem}[\cite{DriesEhrlich01,van2001erratum}]
   The ordinals $\lambda$
  such that $\Nolambda$ is closed under the various fields operations of $\No$ can be
  characterised as follows:
  \begin{itemize}
  \item $\Nolambda$ is an additive subgroup of $\No$ iff
    $\lambda=\omega^\alpha$ for some ordinal $\alpha$.
  \item $\Nolambda$ is a subring of $\No$ iff
    $\lambda=\omega^{\omega^\alpha}$ for some ordinal $\alpha$.
  \item $\Nolambda$ is a subfield of $\No$ iff
    $\omega^\lambda=\lambda$.
  \label{ou}
\end{itemize}
\end{Theorem}
The ordinals $\lambda$ satisfying first (respectively: second) item
are often said to be additively (resp. multiplicatively)
indecomposable but for the sake of brevity we shall just call them
\textbf{additive} (resp. \textbf{multiplicative}). Multiplicative ordinals are exactly
the ordinals $\lambda>1$ such that $\mu \nu <\lambda$ whenever
$\mu,\nu < \lambda$. The ordinal satisfying third item are called
\textbf{$\epsilon$-numbers}. The smallest $\epsilon$-number is usually denoted
by $\epsilon_0$ and is given by
$$\epsilon_0:=\sup\{\omega,\omega^\omega,\omega^{\omega^\omega},\dots\}.$$

\begin{Remark} \label{rqdouze}
	Since rational numbers have length at most $\omega$, we have that if $\lambda$ is multiplicative, then $\Nolt\lambda$ is a divisible group.
\end{Remark}

If $\lambda$ is an $\epsilon$-number, $\Nolambda$ is actually more than only a field: 

\begin{Theorem}[\cite{DriesEhrlich01,van2001erratum}]
Let $\lambda$ be any $\epsilon$-number. Then 
  $\Nolambda$ is a real closed field. 
\end{Theorem}

\subsection{Subfields defined from Hahn's series representation}
 
%
%


As we will often play with exponents of formal power series consided in the Hahn series, we propose to introduce the following notation: 

\begin{Definition}
	If $\lambda$ is an $\epsilon$-number and $\Gamma$ a divisible Abelian group, we denote  $$\SRF{\lambda}{\Gamma}=\HahnFieldOrd\Rbb{\Gamma}\lambda$$
\end{Definition}
As a consequence of Proposition \ref{prop:hahnFieldRealClosed} we have 

\begin{Corollary}
$\SRF\lambda{\Nolt\mu}$ is a real-closed field when $\mu$ is a multiplicative ordinal and $\lambda$ an $\epsilon$-number. 
\end{Corollary}
This fields are somehow the atoms constituting the fields $\Nolambda$.

\thEhrlichquatresept*


\begin{Remark}
The fact that if $\lambda$ is not a regular cardinal, then $\Nolambda \neq\SRF\lambda{\Nolt\lambda}$ can be seen as follows: Suppose that $\lambda$ is not a regular cardinal. This means that we can take some strictly increasing sequence $(\mu_{\alpha})_{\alpha < \beta}$ that is cofinal in $\lambda$ with $\beta<\lambda$. Then $\sum_{\alpha <\beta} \omega^{-\mu_{\alpha}}$ is in $\SRF{\lambda}{\Nolt\lambda}$ by definition, but is not in $\Nolambda$.
\end{Remark}

\section{Exponentiation and logarithm}
\label{sec:expoLn}

\subsection{Gonshor's exponentiation}


The field surreal numbers $\No$ admits an exponential function $\exp$ defined as follows.

\begin{Definition}[{Function $\exp$, \cite[page 145]{gonshor1986introduction}}]
	Let $x = \crotq{x'}{x''}$ be the canonical representation of $x$. We define inductively
	$$\exp x = \crotq {0, \exp(x')[x-x']_{n}, \exp(x'')[x-x'']_{2n+1}} {\f{\exp(x')}{[x'-x]_{2n+1}}, \f{\exp(x'')}{[x''-x]_{2n+1}}}$$
	where $n$ ranges in $\Nbb$ and
	$$[x]_{n}  = 1+ \frac{x}{1!} + \dots + \frac{x}{n!},$$
	with the further convention that the expressions containing terms of the form
	$[y]_{2n+1}$ are to be considered only when $[y]_{2n+1} > 0$.
\end{Definition}
It can be shown that the function $\exp$ is a surjective homomorphism from $(\No, +)$ to $(\No^{>0}, ·)$ which extends $\exp$ on $\R$ and makes $(\No, +, ·, \exp)$ into an elementary extension of $(\R,+,·,\exp)$ (see \cite[Corollaries 2.11 and 4.6]{van1994elementary}, \cite{DriesEhrlich01} and \cite{ressayre1993integer}). As $\exp$ is surjective, and from its properties, it can be shown that it has some inverse $\ln : \No^{>0} \to \No$ (called logarithm).

\begin{Definition}[Functions $\log$, $\log_{n}$, $\exp_{n}$]
 Let $\ln : \No^{>0} \to \No$ (called logarithm) be the inverse of $\exp$. We let $\exp_{n}$ and $\ln_{n}$ be the $n$-fold iterated compositions of $\exp$ and $\ln$ with
themselves.
\end{Definition}

We recall some other basic properties of the exponential functions:

\begin{Theorem}[{\cite[Theorems 10.2, 10.3 and 10.4]{gonshor1986introduction}}]
	\label{thm:expAppreciables}
	For all $r\in\Rbb$ and $\epsilon$ infinitesimal, we have 
	$$\exp r = \Sum{k=0}\infty{\f{r^k}{k!}} \qandq \exp \epsilon = \Sum{k=0}{\infty}\f{\epsilon^k}{k!} \qandq
		\exp(r+\epsilon) = \exp(r)\exp(\epsilon) = \Sum{k=0}{\infty}\f{(r+\epsilon)^k}{k!}$$
	Moreover for all purely infinite number $x$,
		$$\exp(x+r+\epsilon) = \exp(x)\exp(r+\epsilon)$$
\end{Theorem}

\begin{Proposition}[{\cite[Theorem 10.7]{gonshor1986introduction}}]
	\label{prop:formeExpXPurelyInfiniteOmegaA}
	If $x$ is purely infinite, then $\exp x=\omega^a$ for some surreal number $a$. 
\end{Proposition}
More precisely:
\begin{Proposition}[Function $g$, {\cite[Theorem 10.13]{gonshor1986introduction}}]
	\label{prop:formeExpXPurelyInfiniteOmegaAg}
	If $x$ is purely infinite, \textit{i.e.} $x=\Sum{i<\nu}{}r_i\omega^{a_i}$ with $a_i>0$ for all $i$, then 
	$$\exp x=\omega^{\Sum{i<\nu}{}r_i\omega^{g(a_i)}},$$
	for some function $g: \No^{>0 }\to \No$. 	Function $g$ satisfies for all $x$, 
	$$g(x) = \crotq{c(x),g(x')}{g(x'')}$$
	where $c(x)$ is the unique number such that $\omega^{c(x)}$ and $x$ are in the same Archimedean class \cite[Thm. 10.11]{gonshor1986introduction} (i.e. such that 
	$x\asymp\omega^{c(x)}$),  where $x'$ ranges over the lower non-zero prefixes of $x$ and $x''$ over the upper prefixes of $x$.
\end{Proposition}

\subsection{About some properties of function $g$}

\begin{Proposition}[{\cite[Theorem 10.14]{gonshor1986introduction}}]
	\label{prop:gord}
	If $a$ is an ordinal number then
	$$ g(a)=\begin{accolade}
		a+1 & \text{if } \lambda\leq a <\lambda+\omega \text{ for some }\epsilon\text{-number }\lambda\\
		a& \text{otherwise}
	\end{accolade}$$
\end{Proposition}

Note that in the previous proposition, $a\neq 0$ since $g$ is defined only for positive elements.

\begin{Proposition}[{\cite[Theorem 10.15]{gonshor1986introduction}}]
	\label{prop:gMonomeInfinitesimal}
	Let $n$ be a natural number and $b$ be an ordinal. We have
	$ g(2^{-n}\omega^{-b})=-b+2^{-n}.$
\end{Proposition}

\begin{Proposition}[{\cite[Theorems 10.17, 10.19 and 10.20]{gonshor1986introduction}}]
	If $b$ is a surreal number such that for some $\epsilon$-number $\epsilon_i$, some ordinal $\alpha$ and for all natural number $n$,
	$\epsilon_i+n<b<\alpha<\epsilon_{i+1}$, then $g(b)=b$. This is also true if there is some ordinal $\alpha<\epsilon_0$ such that for all natural number $b$, $n\omega^{-1}<b<\alpha <\epsilon_0$.
\end{Proposition}

\begin{Proposition}[{\cite[Theorem 10.18]{gonshor1986introduction}}]
	If $\epsilon\leq b\leq\epsilon + n$ for some $\epsilon$-number $\epsilon$ and some integer $n$. In particular, the sign expansion of $b$ is the sign expansion of $\epsilon$ followed by some sign expansion $S$. Then, the sign expansion of $g(b)$ is the sign expansion of $\epsilon$ followed by a $+$ and then $S$. In particular, $g(b)=b+1$. 
\end{Proposition}

It is possible to bound the length of $g(a)$ depending on the length of $a$. 

\begin{Lemma}[{\cite[Lemma 5.1]{DriesEhrlich01}}]
	\label{lem:lengthGA}
	For all $a\in\Nobf$, $\length{g(a)}\leq \length a +1$.
\end{Lemma}

The function $g$ has a inverse function, $h$ defined as follows

$$h(b) = \crotq{0,h(b')}{h(b''),\f{\omega^b}{n}}$$
This expression is uniform (see \cite{gonshor1986introduction}) and then does not depend of the expression of $b$ as $\crotq{b'}{b''}$.

\begin{Corollary}
	\label{cor:hmord}
	If $a$ is an ordinal number then $h(-a)=\omega^{-a-1}$.
\end{Corollary}

\begin{proof}
	It is a direct consequence of Proposition \ref{prop:gMonomeInfinitesimal} and the fact that $h=g^{-1}$.
\end{proof}

As for $g$, we can bound the length of $h(a)$ in function of the length of $a$.

\begin{Lemma}[{\cite[Proposition 3.1]{aschenbrenner:hal-02350421}}]
	\label{lem:lengthH}
	For all $a\in\Nobf$, $\length{h(a)}\leq\omega^{\length a +1} $
\end{Lemma}
We will also prove another lemma, Lemma \ref{lem:lengthOmegaGA}, that looks like the previous lemma but that is better in many cases but not always. To do so we first prove another technical lemma.

\begin{Lemma}
	\label{lem:gomegacCasSpec}
	For all $c$, denote $c_+$ the surreal number whose signs sequence is the one of $c$ followed by a plus. Assume $g(a)<c$ for all $a\sqsubset\omega^c$ such that $0<a<\omega^c$. Then $g(\omega^c)$ is $c_+$ if $c$ does not have a longest prefix greater than itself, otherwise, $g(\omega^c)=c''$ where $c''$ is the longest prefix of $c$ such that $c''>c$.
\end{Lemma}

\begin{proof}
	By induction on $c$:
	
	\begin{itemize}
		\item For $c=0$, $g(\omega^0)=g(1)=1$ whose signs sequence is indeed the one of $0$ followed by a plus.
		
		\item Assume the property for $b\sqsubset c$. Assume $g(a')<c$ for all $a'\sqsubset\omega^c$ such that $0<a'<\omega^c$. Then,
		$$
			g(\omega^c)=\crotq{c}{g(a'')}
		$$
		where $a''$ ranges over the elements such that $a''\sqsubset \omega^c$ and $a''>\omega^c$.
		
		\begin{itemize}
			\item First case: $c$ has a longest prefix $c_0$ such that $c_0>c$. Then, for all $a''$ such that $a''\sqsubset \omega^c$ and $a''>\omega^c$, $a''\succeq \omega^{c_0}$, hence $g(a'')> c_0$. Since $c<c_0<g(a'')$, the simplicity property ensures $g(\omega^c)\sqsubseteq c_0\sqsubset c$. Then $g(\omega^c)$ is some prefix $c''$ of $c$, greater than $c$. 	We look at $\omega^{c''}$. Notice that for all $b\sqsubset c''$ is such that $0<b<c''$, $b\sqsubset c$ and $b<c$, hence $g(b)<c<c''$. Therefore we can apply the induction hypothesis to $c''$ and $g(\omega^{c''})$ is $c''_+$ if the signs sequence of $c$ does not end with only minuses, otherwise, $g(\omega^{c''})$ is the last (strict) prefix of $c''$ greater than $c''$. 
			
			\begin{itemize}
				\item First subcase: $g(\omega^{c''})=c''_+$. If there is some $b$ such that $c''\sqsubset b\sqsubset c$ and $b>c$, then $g(\omega^b)$ is a prefix of $g(\omega^c)=c''$. But, $c''=g(\omega^c)<g(\omega^b)<g(\omega^{c''})=c''_+$. Then $c''$ must be a strict prefix of $g(\omega^b)$ which is a contradiction. Then $c''$ is indeed the last strict prefix of $c$ greater than $c$.
				
				\item Second subcase:  $g(\omega^{c''})$ is the last (strict) prefix of $c''$ greater than $c''$. If there is some $b$ such that $c''\sqsubset b\sqsubset c$ and $b>c$, then $g(\omega^b)$ is a prefix of $g(\omega^c)=c''$. Since $g(b)<g(c'')$, $g(b)$ is prefix of $c''$ smaller than $c''$. But this contradicts the fact that $g(\omega^b)>g(\omega^c)=c''$. Therefore, $c''$ is the last prefix of $c$ greater than $c$.
			\end{itemize}

			\item Second case: $c$ does not have a longest prefix greater than $c$. Then,
			$$
				g(\omega^c) = \crotq c{g(\omega^{c''})}
			$$
			where $c''$ ranges over the prefixes of $c$ greater than $c$. Let $d\sqsubset c$ such that $d>c$. Then there is $d_1$ or minimal length such that $d\sqsubset d_1\sqsubset c$ and $d_1>c$. By minimality of $d_1$, $d$ is the longest prefix of $d_1$ greater than $d_1$. As in the first case, we can apply the induction hypothesis on $d_1$ and get $g(\omega^{d_1})=d$. Therefore, again by induction hypothesis,
			$$
				g(\omega^c)=\crotq c{c'',c''_+} = \crotq c{c''}
			$$
			where $c''$ ranges over the prefixes of $c$ greater than $c$. We finally conclude that $g(\omega^c)=c_+$.
		\end{itemize}
	\end{itemize}
\end{proof}

In the following we denote $\oplus$ the usual addition over the ordinal numbers and~$\otimes$ the usual product over ordinal numbers.

\begin{Lemma}
	\label{lem:lengthOmegaGA}
	For all $a>0$, $\length a\leq\length{\omega^{g(a)}}\otimes(\omega+1)$.
\end{Lemma}

\begin{proof}
	We proceed by induction on $\length a$. 
	\begin{itemize}
		\item For $a=1$, $g(a)=1$ and we indeed have $1\leq \omega^2$.
		
		\item Assume the property for all $b\sqsubset a$. Let $c$ such that $\omega^c\asymp a$. Then 
		$$
			g(a)=\crotq{c,g(a')}{g(a'')}
		$$
		We split into two cases:
		\begin{itemize}
			\item If there is some $a_0\sqsubset a$ such that $a_0<a$ and $g(a_0)\geq c$ then
			$$
			g(a)=\crotq{g(a')}{g(a'')}
			$$
			and if $S$ stand for the signs sequence such that $a$ is the signs sequence of $a_0$ followed by $S$, $g(a)$ is the signs sequence of $g(a_0)$ followed $S$. Let $\alpha$ the length of $S$. Therefore using Theorem \ref{thm:serieToSignExp},
			$$
			\length{\omega^{g(a)}}\geq\length{\omega^{g(a_0)}}\oplus(\omega\otimes\alpha)
			$$
			and then,
			\begin{align*}
				\length{\omega^{g(a)}}\otimes(\omega+1) &\geq \length{\omega^{g(a_0)}}\otimes\omega\oplus \length{\omega^{g(a_0)}}\oplus\alpha \\
				&\geq\length{\omega^{g(a_0)}}\otimes(\omega+1)\oplus\alpha
			\end{align*}
			and by induction hypothesis on $a_0$,
			$$
			\length{\omega^{g(a)}}\otimes(\omega+1)\geq\length{a_0}\oplus\alpha=\length a
			$$
			
			\item Otherwise, for any $a_0\sqsubset a$ such that $a_0<a$, $g(a_0)<c$. Therefore,
			$$ g(a)=\crotq{c}{g(a'')}$$
			Also, since $a>0$, we can write the signs sequence of $a$ as the one of $\omega^c$ followed by some signs sequence $S$. If $S$ contains a plus, then there is a prefix of $a$, $a_0$ such that $a_0<a$ and still $a_0\asymp\omega^c$ and then $g(a_0)>c$ what is not the case by assumption. Then, $S$ is a sequence of minuses. If $S$ is not the empty sequence, let $\alpha$ be the length of $S$. Then the signs sequence of $g(a)$ is the one of $g(\omega^c)$ followed by $S$. Hence,
			$$
				\length{\omega^{g(a)}}\geq \length{\omega^{g(\omega^c)}}\oplus(\omega\otimes\alpha)
			$$ 
			As in the previous case, but using the induction hypothesis on $\omega^c$,
			$$
				\length{\omega^{g(a)}}\otimes(\omega+1)\geq\length{\omega^c}\oplus\alpha=\length a
			$$
			Now if $S$ is the empty sequence, $a=\omega^c$. Applying Lemma \ref{lem:gomegacCasSpec} to $c$ we get that either $g(a)=c_+$ or $g(a)$ is the last prefix of $c$ greater than $c$. If the first case occurs then $a$ is a prefix of $\omega^{g(a)}$ and then $\length{\omega^{g(a)}}\geq\length a$. Now assume that the second case occurs. Then for any $b$ such that $g(a)\sqsubset b\sqsubset c$, $b<c$.
			If for all $b'\sqsubset b$ such that $b'<b$, $g(b')<b$, then Lemma \ref{lem:gomegacCasSpec} applies. Since $b$ has a last prefix greater than itself, $g(a)$, $g(\omega^b)=g(a)$ and we reach a contradiction since $b<c$ and therefore $\omega^b<\omega^c=a$. Then for all  $b$ such that $g(a)\sqsubset b\sqsubset c$, there is some $b'\sqsubset b$, $b'<b$ such that $g(\omega^{b'})>b$. Since the signs sequence of $b$ consists in the one of $g(a)$ a minus and then a bunch of pluses, and since $g(\omega^{b'})$ must also a a prefix of $c$, $g(\omega^{b'})\sqsubseteq g(a)\sqsubset b$. Therefore to ensure $g(b')>b$, we must have $g(\omega^{b'})\geq g(a)$. Since $\omega^{b'}$ is a prefix of $a$ lower than $a$, it is a contradiction. Therefore, there is no $b$ such that $g(a)\sqsubset b\sqsubset c$ and $b<c$, and finally, the signs sequence of $c$ is the one $g(a)$ followed by a minus. In particular, $g(a)$ and $c$ have the same amount of pluses, say $\alpha$. Then, using Theorem \ref{thm:serieToSignExp},
			\begin{align*}
				\length a &= \length{\omega^{g(a)}}\oplus\omega^{\alpha+1}\\
				&\leq \length{\omega^{g(a)}}\oplus \length{\omega^{g(a)}}\otimes\omega = \length{\omega^{g(a)}}\otimes\omega\\
				&\leq \length{\omega^{g(a)}}\otimes(\omega+1)
			\end{align*}
		\end{itemize}
		The induction principle concludes.
	\end{itemize}
\end{proof}

\begin{Corollary}
	\label{cor:lengthOmegaGA}
	For all $a>0$ and for all multiplicative ordinal greater than $\omega$, if $\length a\geq\mu$, then $\length{\omega^{g(a)}}\geq\mu$.
\end{Corollary}

\begin{proof}
	Assume the that $\length{\omega^{g(a)}}<\mu$. Then using Lemma \ref{lem:lengthOmegaGA}, $\mu\leq \length{\omega^{g(a)}}\otimes(\omega+1)$. Since $\mu$ is a multiplicative ordinal greater than $\omega$, we have $\omega+1<\mu$. $\mu$ is a multiplicative ordinal, hence $\length{\omega^{g(a)}}\otimes(\omega+1)<\mu$ and we reach a contradiction.
\end{proof}

\subsection{Gonshor's logarithm}

We already know that a logarithm exist over positive surreal numbers. Nevertheless we were very elliptical and we now get deeper into it.

\begin{Definition}
	For a surreal number $a$ in canonical representation $a=\crotq{a'}{a''}$, we define
	$$
		\ln\omega^a=\crotq{\begin{array}{c}
				\enstq{\ln\omega^{a'}+n}{\begin{array}{c}n\in\Nbb\\ a'\sqsubset a\\ a'<a\end{array}}\\
				\enstq{\ln\omega^{a''}-\omega^{\frac{a''-a}n}}{\begin{array}{c}n\in\Nbb\\ a''\sqsubset a\\ a<a''\end{array}}
			\end{array}}{\begin{array}{c}
				\enstq{\ln\omega^{a''}-n}{\begin{array}{c}n\in\Nbb\\ a''\sqsubset a\\ a<a''\end{array}}\\
				\enstq{\ln\omega^{a'}+\omega^{\frac{a-a'}{n}}}{\begin{array}{c}n\in\Nbb\\ a'\sqsubset a\\ a'<a\end{array}}
			\end{array}}
	$$
\end{Definition}
As often with these definition, the uniformity property holds.
\begin{Lemma}[{\cite[Lemma 10.1]{gonshor1986introduction}}]
	The definition of $\ln\omega^a$ does not require $a$ in canonical representation.
\end{Lemma}

\begin{Proposition}[{\cite[Theorem 10.8]{gonshor1986introduction}}]
	\label{prop:lnOmegaA}
	For all surreal number $a$, $\ln\omega^a$ is purely infinite.
\end{Proposition}
Purely infinite numbers are a special case in the definition of the exponential function. We can state the previous definition of $\ln$ is consistent with the one of $\exp$.

\begin{Theorem}[{\cite[Theorem 10.9]{gonshor1986introduction}}]
	\label{thm:expLnOmegaA}
	For all surreal number $a$, $\exp\ln\omega^a=\omega^a$.
\end{Theorem}

\begin{Theorem}[{\cite[Theorem 10.12]{gonshor1986introduction}}]
	\label{thm:lnOmegaOmegaA}
	For all surreal number $a$, $\ln\omega^{\omega^a}=\omega^{h(a)}$.
\end{Theorem}
The above theorem is not actually stated like this in \cite{gonshor1986introduction} but this statement follows from the proof there.

As a consequence of Theorems \ref{thm:expLnOmegaA} and \ref{thm:lnOmegaOmegaA} and Propositions \ref{prop:lnOmegaA} and \ref{prop:formeExpXPurelyInfiniteOmegaAg}, we have 

\begin{Corollary}
	For all surreal number $a=\aSurreal$, we have
	$$
		\ln\omega^a = \Sumlt i\nu r_i\omega^{h(a_i)}
	$$ 
\end{Corollary}

Finally, since for appreciable numbers $\exp$ is defined by its usual serie, $\ln(1+x)$ is also defined by its usual serie when $x$ in infinitesimal. More precisely,

\begin{Definition}
	\label{def:lnAppreciables}
	For $x$ an infinitesimal, 
	$$
		\ln(1+x) = \Sum{i=1}{\infty}\frac{(-1)^{i-1}x^i}{i}
	$$
\end{Definition}
And thanks to Theorem \ref{thm:expAppreciables},

\begin{Corollary}
	Let $a=\aSurreal$ a positive surreal number. Then
	$$
		\ln a = \ln\omega^{a_0} + \ln r_0 + \ln\pa{1+\Sumlt{1\leq i}\nu \frac{r_i}{r_0}\omega^{a_i-a_0}}
	$$
	where the last term is defined in Definition \ref{def:lnAppreciables}.
\end{Corollary}

\subsection{Stability of $\Nolambda$ by exponential and logarithm}

We first recall some result by van den Dries and Ehrlich.

\begin{Lemma}[{\cite[Lemmas 5.2, 5.3 and 5.4]{DriesEhrlich01}}]
	\label{lem:lengthExpLog}
	For all surreal number $a\in\Nobf$,
	\begin{itemize}
		
		\item $\length{\exp a}\leq\omega^{\omega^{2\length a\oplus3}}$
		
		\item $\length{\ln\omega^a} \leq \omega^{4\omega\length{a}\length{a}}$
		
		\item $\length{\ln a}\leq\omega^{\omega^{3\length a\oplus3}}$
	\end{itemize}
\end{Lemma}

\begin{Corollary}[{\cite[Corollary 5.5]{DriesEhrlich01}}]
	\label{cor:NolambdaStableExpLn}
	For $\lambda$ an $\epsilon$-number, $\Nolt\lambda$ is stable under $\exp$ and $\ln$.
\end{Corollary}

We have (already stated in the introduction, recalled here for readability and for stating the proof). 

\thNoltStableExpLn*

\begin{proof}
	Using  Theorem \ref{thm:structureNolambda} we already know that $\Nolambda$ is a field is an only if $\lambda$ is an $\epsilon$-number. Corollary \ref{cor:NolambdaStableExpLn} ensure that if $\lambda$ is an $\epsilon$-number, $\Nolambda$ is stable under exponential and logarithm. The last thing to prove is that if $\lambda$ is not an $\epsilon$-number, then $\Nolambda$ is not stable under one of theses functions. Then, let $\lambda$ be an ordinal which is not an $\epsilon$-number. Write in in the Cantor normal form as
	
	$$ \lambda=\Sum{i=0}{n}\omega^{\alpha_i}n_i$$
	with $n$ a natural number as well as the coefficients $n_i$ and $(\alpha_0,\dots,\alpha_n)$ being a finite decreasing sequence of ordinals. Since $\lambda$ is not an $\epsilon$-number, $\alpha_0<\lambda$. In particular, $\alpha_0\in\Nolambda$. If $\lambda=\omega^{\alpha_0}$ then Lemma \ref{lem:lengthExpLog} give that $\exp(\alpha_0)$ has length at least $\lambda$. Therefore $\Nolambda$ is not stable under exponential. If not, $\lambda>\omega^{\alpha_0}$ and then $\exp\left(\omega^{\alpha_0}\right)=\omega^{\omega^{g(\alpha_0)}}$. By Proposition \ref{prop:gord}, $g(\alpha_0)$ is an ordinal number and 
	
	$$ g(\alpha_0)=\begin{accolade}
		\alpha_0+1 & \text{if } \lambda'\leq \alpha_0 <\lambda'+\omega \text{ for some }\epsilon\text{-number }\lambda'\\
		\alpha_0& \text{otherwise}
	\end{accolade}$$
	
	In both cases we have $\alpha_0<\omega^{g(\alpha_0)}$. Therefore, by Lemma \ref{lem:lengthExpLog}, $\exp\left(\omega^{\alpha_0}\right)$ has length at least $\omega^{\omega^{\alpha_0}}$ which is greater than $\lambda$. 
\end{proof}

\subsection{An instability result of the decomposition of $\Nolambda$}

One point about Theorem \ref{th:Ehrlichquatresept} is that it establishes that $\Nolambda$ can be expressed as an increasing union of fields. However, even if $\Nolt\lambda$ is stable under $\exp$ and $\ln$ (Theorem \ref{thm:NoltStableExpLn}) none of the fields in this union has stability property beyond the fact that they are fields. Indeed, we have the following proposition (already stated in the introduction, recalled here):

\firstinstable*

\begin{proof}
	If $\mu$ is a multiplicative ordinal but not an $\epsilon$-number, $\mu=\omega^{\omega^\alpha}$ for some ordinal $\alpha<\mu$. Since $g(\omega^\alpha)\geq \omega^\alpha$ (Proposition \ref{prop:gord}), Proposition \ref{prop:formeExpXPurelyInfiniteOmegaAg} ensures that $\exp(\omega^\alpha)=\omega^{\omega^{g(\alpha)}}\geq \mu$. Therefore, $\exp(\omega^\alpha)\notin\SRF\lambda{\Nolt\mu}$.

	Now, if $\mu$ is an $\epsilon$-number. Take $x=\Sum{0<i<\mu}{}\omega^{\omega^{-i}}$. Then by Propositions \ref{prop:formeExpXPurelyInfiniteOmegaAg} and \ref{prop:gMonomeInfinitesimal} ensure that $\exp(x)=\omega^{\Sum{0<i<\mu}{}\omega^{g(\omega^{-i})}}=\omega^{\Sum{0<i<\mu}{}\omega^{-i}}$. Since $\mu$ is an $\epsilon$-number, for $i<\mu$, $\omega^{-i}\in\Nolt\mu$ but $\Sum{0<i<\mu}{}\omega^{-i}\notin\Nolt\mu$ (as a consequence of Theorem \ref{thm:serieToSignExp}, the serie having length $\mu$, the length of the surreal number is at least $\mu$). Therefore $x\in\SRF\lambda{\Nolt\mu}$ and $\exp x\notin\SRF\lambda{\Nolt\mu}$.
\end{proof}

The aim of Theorems \ref{thm:SRFGammaUpStableExpLn} and \ref{thm:NolambdaDecompCorpsStables} is to solve this problem, by proposing a new decomposition of $\Nolt\lambda$ as a union of fields that are stable under both exponential and logarithm.

\section{A hierarchy of subfields of $\No$ stable by exponential and logarithm}
\label{sec:proofs}
This section is devoted to prove Theorems \ref{thm:SRFGammaUpStableExpLn}, \ref{thm:NolambdaDecompCorpsStables} and \ref{thm:hierarchieUparrow}. 

\subsection{Hierarchy of fields stable by exponential and logarithm}
We start by Theorem \ref{thm:SRFGammaUpStableExpLn} repeated here for readability:

\thmSRFGammaUpStableExpLn*

We will actually prove a stronger result which is the following proposition.

\begin{Proposition}
	\label{prop:UnionSRFStableExpLn}
	Let $\lambda$ be an $\epsilon$-number and $\suite{\Gamma_i}iI$ be a family of Abelian subgroups of $\Nobf$. Then
	$\RlGI$ is stable under $\exp$ and $\ln$ if and only if $$\Unionin iI\Gamma_i=\Unionin iI\SRF\lambda{g\pa{\pa{\Gamma_i}^*_+}}$$
\end{Proposition}

\begin{proof}
	
	\begin{itemize}
		\item \textbf{Proof of direct implication}: We assume that $\RlGI$ is stable under both exponential and logarithm. Then for any $x=\aSurreal$ a purely infinite number, we have 
		
		$$
			\exp x = \omega^{\Sumlt i\nu r_i\omega^{g(a_i)}}\in\RlGI
		$$
		and therefore 
		
		$$
			\Sumlt i\nu r_i\omega^{g(a_i)}\in\Unionin jI \Gamma_j
		$$
		This being true for any family $(a_i)_{i<\nu}$ of $\Gamma_j$, for any $j\in I$. Hence, $\Unionin iI \SRF\lambda{g\pa{\pa{\Gamma_i}^*_+}}\subseteq\Unionin iI\Gamma_i$.
		
		Conversely, for any $j\in I$ and any $a\in\Gamma_j$ then we have $\ln\omega^a\in\RlGI$. Writing $a=\aSurreal$, we get $\Sumlt i\nu r_i\omega^{h(a_i)}\in\RlGI$. To say it another way,
		$$\exists k\in I\quad \forall i<\nu\qquad h(a_i)\in\Gamma_k$$
		or
		$$\exists k\in I\quad \forall i<\nu\qquad a_i\in g\left(\left(\Gamma_k\right)^*_+\right)$$
		Then, there is some $k\in I$ such that $a\in\SRF\lambda{g\pa{\pa{\Gamma_k}^*_+}}$. Hence, for all $j\in I$, $\Gamma_j\subseteq \Unionin iI \SRF\lambda{g\pa{\pa{\Gamma_i}^*_+}}$. 
		Finally, $\Unionin iI \SRF\lambda{g\pa{\pa{\Gamma_i}^*_+}}\supseteq\Unionin iI\Gamma_i$.
		
		We have both inclusions, then, $\Unionin iI \SRF\lambda{g\pa{\pa{\Gamma_i}^*_+}} = \Unionin iI\Gamma_i$
		
		\item \textbf{Proof of indirect implication}: We assume that $\Unionin iI \SRF\lambda{g\pa{\pa{\Gamma_i}^*_+}} = \Unionin iI\Gamma_i$. We distinguish the proof in several steps
		
		\begin{enumerate}[label=(\roman*)]
			\item\label{it:casExpApp} 
			First take $x= \aSurreal\in\RlGI$ being appreciable, i.e. $a_i\leq 0$ for all $i<\nu$. By definition there is some $j\in I$ such that $x$ is an element of $\SRF\lambda{\Gamma_j}$. Since from Theorem \ref{thm:expAppreciables} for any appreciable number 
			$$\supp\exp x \subseteq \inner{\supp x} $$
			\noindent where $\inner{\supp x}$ is the monoid generated by $\supp x$ in $\Gamma_j$. In particular $\supp\exp x\subseteq\Gamma_j$. Then, Proposition \ref{prop:orderTypeMonoid} ensures that the order type of $\supp\exp x$ is less than $\lambda$. Hence $\exp x \in\RlGI$. 
			
			\item\label{it:casExpInf} 
			
			Let $x=\aSurreal\in\RlGI$ a purely infinite number. Let $j\in I$ such that $x\in\SRF\lambda{\Gamma_j}$ that is that $a_i\in\left(\Gamma_j\right)^*_+$ for all $i<\nu$.  We have $\exp x = \omega^{\Sumlt i\nu r_i\omega^{g(a_i)}}$ and 
			$$
			\Sumlt i\nu r_i\omega^{g(a_i)}\in \SRF\lambda{g\left(\left(\Gamma_j\right)^*_+\right)}
			$$
			
			By assumption, $\SRF\lambda{g\left(\left(\Gamma_j\right)^*_+\right)}\subseteq \Unionin iI\Gamma_i$. Then $\exp x\in \RlGI$.
			
			\item\label{it:casExpG} 
			 We now make use of both Items \ref{it:casExpApp} and \ref{it:casExpInf}. Let $x\in\RlGI$ be arbitrary. Let $x_\infty$ its purely infinite part and $x_a$ its appreciable part. Then $x=x_\infty+x_a$ and $\exp x = \exp(x_\infty)\exp(x_a)$. Using \ref{it:casExpInf} and \ref{it:casExpApp} respectively, we have $\exp x_\infty\in\RlGI$ and $\exp x_a\in\RlGI$. Then since $\RlGI$ is a field, $\exp x\in\RlGI$.
			
			\item\label{it:casLnInf} 
			Similarly to Point \ref{it:casExpApp}, if $x=\aSurreal\in\RlGI$ is infinitesimal, i.e. $a_i<0$ for all $i<\nu$, then $\ln(1+x)=\Sum{k=1}{\infty}\f{x^k}{k}\in\RlGI$
			
			\item\label{it:casLnOmega} 
			Let $a\in\Unionin iI\Gamma_i$. By assumption there is $j\in I$ such that $a\in\SRF\lambda{g\left(\left(\Gamma_j\right)^*_+\right)}$.  Hence, we can write $a=\Sumlt i\nu r_i\omega^{g(a_i)}$ where $\nu<\lambda$ and $a_i\in \left(\Gamma_j\right)^*_+$ for all $i<\nu$. Then, $\ln\omega^a=\aSurreal$. Hence $\ln\omega^a\in\SRF\lambda{\Gamma_j}\subseteq\RlGI$.
			
			\item\label{it:casLnG} 
			Let $x\in\left(\RlGI\right)^*_+$ be arbitrary and write it as $x=r\omega^a(1+\epsilon)$ where $\epsilon$ is infinitesimal, $r$ is a positive real number and $a$ a surreal number. Then, 
			$\ln x = \ln\omega^a+\ln r+\ln(1+\epsilon)$. Then since $\RlGI$ is a field, $\exp x\in\RlGI$. Using 
			\ref{it:casLnOmega} and \ref{it:casLnInf} respectively, we have $\ln \omega^a\in\RlGI$ and $\ln(1+\epsilon)\in\RlGI$. Then since $\RlGI$ is a field containing $\R$, $\ln x\in\RlGI$.
			
			\item Item \ref{it:casExpG} proves that $\RlGI$ is stable under exponential and Item \ref{it:casLnG} that $\RlGI$ is stable under logarithm. This is what was announced.
			
		\end{enumerate}
		
	\end{itemize}
\end{proof}

We are now ready to prove the theorem. We use the notations of Definition~\ref{def:uparrow}.

\begin{proof}[Proof of Theorem \ref{thm:SRFGammaUpStableExpLn}]
	We write $\Gamma^{\uparrow\lambda}=(\Gamma_\beta)_{\beta<\gamma_\lambda}$. Using Proposition \ref{prop:UnionSRFStableExpLn}, we just need to show
	$$\Unionlt \beta{\gamma_\lambda}\Gamma_\beta = \Unionlt \beta{\gamma_\lambda}\SRF\lambda{g\pa{\pa{\Gamma_\beta}^*_+}}$$
	\begin{itemize}
		\item[\CSsubset] Let $x\in\SRF\lambda{g\pa{\pa{\Gamma_\beta}^*_+}}$. Let $n<\gamma_\lambda$ minimal such that $\nu(x)<e_n$. Then $x\in\Gamma_{\max(n,\beta)}$.
		
		\item[\CNsubset] Let $x\in\Gamma_\beta$. Write $x=\aSurreal$. We also have $x=\Sumlt i\nu r_i\omega^{g(h(a_i))}$ and $h(a_i)\in\Gamma_{\beta+1}$.  Then $x\in\SRF\lambda{g\pa{\pa{\Gamma_{\beta+1}}^*_+}}$.
	\end{itemize}
\end{proof}

As a final note of this section we notice that a consequence of Proposition \ref{prop:UnionSRFStableExpLn} is also the following:

\begin{Corollary}
	Let $\lambda$ be an $\epsilon$-number and $\Gamma$ be an abelian subgroup of $\Nobf$. Then
	$\RlG$ is stable under $\exp$ and $\ln$ if and only if $\Gamma=\SRF\lambda{g\pa{\Gamma^*_+}}$.
\end{Corollary}
This result is quite similar to Theorem \ref{thm:SRFGammaUpStableExpLn} but in the particular very particular case where $\Unionin G{\Gamma^{\uparrow\lambda}}G=\Gamma$. This apply for instance when $\Gamma=\{0\}$. In this case, we get $\RlG=\R$. If $\lambda$ is a regular cardinal we get an other example considering $\SRF\lambda\Gamma=\Gamma=\Nolt\lambda$.

\subsection{Proof of Theorem \ref{thm:NolambdaDecompCorpsStables}}

To prove the Theorem \ref{thm:NolambdaDecompCorpsStables}, we first prove a proposition to ensure inclusion of the fields in the union.

\begin{Proposition}
	\label{prop:stableExpLnContenuNoLambda}
	Let $\lambda$ be an $\epsilon$-number and $\mu<\lambda$ an additive (or multiplicative) ordinal. If $\Gamma\subseteq\Nolt\mu$ then $\RlGup\subseteq\Nolt\lambda$ 
\end{Proposition}

\begin{proof}
	Write $\Gamma^{\uparrow\lambda}=\suitelt{\Gamma_\beta}\beta{\gamma_\lambda}$. What we have to prove is that for all $i<\gamma_\lambda$, $\Gamma_i\subseteq\Nolt{\mu_i}$ for some $\mu_i<\lambda$. We will even prove that $\mu_i=e_{k\oplus 2\otimes i}$ works for some fixed ordinal $k$. We prove it by induction on $i$.
	
	\begin{itemize}
		\item For $i=0$, $\mu_0=e_k$ with $k$ the least ordinal such that $\mu\leq e_k$ works. 
		
		\item Assume $i=j+1$ and that the property is true for $j$. Therefore $\Gamma_{i}$ is the group generated by $\Gamma_j$, $\SRF{e_j}{g\pa{(\Gamma_j)^*_+}}$ and $\enstq{h(a_k)}{\Sumlt k\nu r_k\omega^{a_k}\in\Gamma_j}$. Thanks to the induction hypothesis and Lemma \ref{lem:lengthGA}, $g\pa{(\Gamma_j)^*_+}\subseteq \Nolt{\mu_j}$, since $\mu_j$ is an additive ordinal. Hence, thanks to Lemma \ref{lem:lengthOmegaA}, $\SRF{e_j}{g\pa{(\Gamma_j)^*_+}}\subseteq \Nolt{\omega^{\mu_j\otimes\omega}\otimes e_j}$. Finally, from Lemmas \ref{lem:lengthH} and \ref{lem:lengthOmegaA}, $h(a_k)\in\Nolt{\omega^{\mu_j}}$. Thus, $\Gamma_i\subseteq \Nolt{\omega^{\mu_j\otimes\omega}\otimes e_j}$. Since  $\omega^{\mu_j\otimes\omega},e_j<e_{k\oplus2\otimes i}$, and $e_{k\oplus2\otimes i}$ is multiplicative,  taking, $\mu_i=e_{k\oplus 2\otimes i}$ works.
		
		\item If $i<\gamma_\lambda$ is a limit ordinal, for all $j<i$, $\lambda > e_{k\oplus2\otimes i}>e_{k\oplus 2\otimes j}$. Then, by the induction hypothesis on all $j<i$, $\Gamma_i\subseteq\Nolt{e_{k\oplus2\otimes i}}$.
	\end{itemize}
\end{proof}

With the previous proposition, we have all what we need to prove Theorem \ref{thm:NolambdaDecompCorpsStables}, that we repeat here for readabiligy:

\thmNolambdaDecompCorpsStables*

\begin{proof}[Proof of Theorem \ref{thm:NolambdaDecompCorpsStables}]
	Using Theorem \ref{th:Ehrlichquatresept}, we know that 
	$$\Nolambda=\Unionin\mu{\enstq{\mu<\lambda}{\mu\text{ additive ordinal}}}\SRF\lambda{\Nolt\mu}$$ 
	By definition of  $\SRF\lambda{{\Nolt\mu}^{\uparrow\lambda}}$, it must contain $\SRF\lambda{\Nolt\mu}$ and then
	$$\Nolambda\subseteq \Unionin\mu{\enstq{\mu<\lambda}{\mu\text{ additive ordinal}}}\SRF\lambda{{\Nolt\mu}^{\uparrow\lambda}}$$
	On the other hand, applying Proposition \ref{prop:stableExpLnContenuNoLambda} gives
	$$\Unionin\mu{\enstq{\mu<\lambda}{\mu\text{ additive ordinal}}}\SRF\lambda{{\Nolt\mu}^{\uparrow\lambda}} \subseteq \Nolt\lambda$$
	and this concludes the proof.
\end{proof}

\subsection{Strictness of the Hierarchy} 

The hierarchy is strict (the theorem is reformulated here to help readability):

\thmhierarchieUparrow*

To prove the theorem we need to use the concept of log-atomic number that were introduced by Berarducci and Mantova in \cite{berarducci2018surreal}. 

\begin{Definition}[Log-atomic number, {\cite[Definition 5.1]{berarducci2018surreal}}]
	\label{def:logAtomique}
	A surreal number $x$ is said \textbf{log-atomic} if and only if for all natural number $n\in\Nbb$ there is some surreal number $a_n$ such that $\ln_n x = \omega^{a_n}$. The class of log-atomic number is denoted $\Lbb$.
\end{Definition}
A first example of such a number is $\omega$. Indeed, for any natural number $n$, it is possible to prove that $\ln_n\omega = \omega^{\frac1{\omega^n}}$. More generally, from Corollary \ref{cor:hmord} we get that iif $\mu=\omega^{\omega^{-\alpha}}$ for some ordinal $\alpha$, then $\ln_n\mu=\omega^{\omega^{-\alpha-n}}$ and $\mu$ is also a log-atomic number. 

\begin{Remark}
	If $x\in\Lbb$ then for all natural number $n\in\Nbb$, we have $\exp_nx\in\Lbb$ and $\ln_nx\in\Lbb$.
\end{Remark}

We expose the definition of \textbf{path} adapted to surreal numbers given by Berarducci and Mantova in \cite{berarducci2018surreal}. It was originally introduced by Schmeling in \cite{schmeling2001corps} for transseries.

\begin{Definition}[Path {\cite{schmeling2001corps}, \cite{berarducci2018surreal}}]
	A \textbf{path} is a function $P$ from $\Nbb$ to terms such that: 
	\begin{itemize}
		\item For $i\in\Nbb$, $P(i)$ is a term, i.e. there are some non-zero real number $r_i$ and some surreal number $a_i$ such that $P(i)=r_i\omega^{a_i}$.
		\item For $i\in\Nbb$, $P(i+1)$ is a term of $\ln\omega^{a_i}$ where $a_i$ is the unique surreal number such that $\omega^{a_i}\asymp P(i)$.
	\end{itemize}
	If $P$ is a path and $x=\aSurreal$ we say that \textbf{$P$ is a path of $x$} if $P(0)=r_i\omega^{a_i}$ for some $i<\nu$. We denote $\Pcal(x)$ the set of path of $x$.
\end{Definition}
Notice that if for some $i\in\Nbb$ $P(i)$ is a log-atomic number, then for all natural number $n\in\Nbb$, $P(i+n)=\ln_n P(i)$ is forced. In particular there is no more choice possible for $P$.

We now state that the construction of $\Gamma^{\uparrow\lambda}$ does not create new log-atomic numbers (up to some iteration of $\exp$ or $\ln$).  After that, we will prove that we do introduce new log-atomic numbers when going trough the hierarchy.

\begin{Lemma}
	\label{lem:LogAtomicDesGammai}
	Write $\Gamma^{\uparrow\lambda}=\suitelt{\Gamma_\beta}\beta{\gamma_\lambda}$, and let 
	$$L=\enstq{\exp_n x, \ln_nx}{x\in\Lbb,\quad n\in\Nbb,\quad \exists y\in\RlG\ \exists P\in\Pcal(y)\ \exists k\in\Nbb\quad P(k)=x}$$
	we have for all $i<\gamma_\lambda$,
	$$L=\enstq{\exp_n x, \ln_nx}{x\in\Lbb,\quad n\in\Nbb,\quad \exists y\in\SRF\lambda{\Gamma_i}\ \exists P\in\Pcal(y)\ \exists k\in\Nbb\quad P(k)=x}$$
\end{Lemma}

\begin{proof} We prove it by induction on $i$.
	\begin{itemize}
		\item For $i=0$, $\Gamma_0=\Gamma$ then there is noting to prove.
		
		\item Assume the property for some ordinal $i<\gamma_\lambda$.  We prove it for $i+1$.
		\begin{itemize}
			\item[\CNsubset] Trivial since $\SRF\lambda{\Gamma_i}\subseteq\SRF\lambda{\Gamma_{i+1}}$.
			\item[\CSsubset] Let $x\in\Lbb$, $y\in\SRF\lambda{\Gamma_{i+1}}$, $P\in\Pcal(y)$ and $k\in\Nbb$ such that $P(k)=x$. Write $P(0)=r\omega^a$ a term of $x$ with $a\in\Gamma_{i+1}$. Then $a$ can be written
			$$a = u+v+\Sum{j=1}k \sigma_jh(w_j)$$
			with $u\in\Gamma_i$, $v\in \SRF{e_i}{g\pa{(\Gamma_i)_+^*}}$, $\sigma_j\in\{-1,1\}$ and $w_j\in \Gamma_i$.
			By definition of a path, $P(1)$ is a purely infinite term of 
			$$\ln\omega^a=\ln\omega^u+\ln\omega^v+\Sum{j=1}k\ln\omega^{\sigma_jh(w_j)} = \ln\omega^u+\ln\omega^v+\Sum{j=1}k\sigma_j\ln\omega^{h(w_j)}$$
			Then, up to a real factor $s$, $P(1)$ is a term of either $\ln\omega^u$ or $\ln\omega^v$ or $\ln\omega^{h(w_j)}$ for some $j$.
			
			\begin{itemize}
				\item \underline{Case 1}: $sP(1)$ is a purely infinite term of $\ln\omega^{u}$. Then the function
				$$Q(m)=\begin{cases}
						\omega^u & \text{if }m=0 \\ sP(1) & \text{if }m=1 \\ P(m) & \text{if }m>1
					\end{cases}$$
				is a path of $\omega^u\in\SRF\lambda{\Gamma_i}$. Then, if $m\geq\max(2,n)$, $Q(m) = P(m)=\ln_{m-n}(x)$ then 
				for all $n\in\Nbb$, 
				$$\exp_nx,\ln_nx\in\enstq{\exp_n x, \ln_nx}{\begin{array}{c}
						x\in\Lbb,\quad n\in\Nbb,\\
						\exists y\in\SRF\lambda{\Gamma_{i+1}}\ \exists P\in\Pcal(y)\ \exists k\in\Nbb\quad P(k)=x
					\end{array}}$$
				
				\item \underline{Case 2}: $sP(1)$ is a purely infinite term of $\ln\omega^v$. Write $v=\Sumlt i{\nu'}s_i\omega^{g(b_i)}$ where $b_i\in \Gamma_k$. Again, the function
				$$Q(m)=\begin{cases}
						sP(1) & \text{if }m=0 \\ P(m+1) & \text{if }m>0
					\end{cases}$$
				is a path of $\ln\omega^v=\Sumlt{i}{\nu'}s_i\omega^{b_i}\in\SRF\lambda{\Gamma_i}$. Then,if $m\geq\max(1,n-1)$, $Q(m)=\ln_{m-n+1}x\in L$ and we are done.
				
				\item \underline{Case 3:} $sP(1)$ is a purely infinite term of $\ln\omega^{h(w_j)}$. From the definition of $w_j$, there is $s'\in\Rbb^*$ such that $s'\omega^{w_j}$ is a term of some element of $y\in\Gamma_n$. Then $s'\omega^{h(w_i)}$ is a purely infinite term of $\ln\omega^y$. Then the function
				$$Q(m)=\begin{cases}
						\omega^y & \text{if }m=0 \\ s'\omega^{h(w_i)} & \text{if }m=1 \\ sP(1) & \text{if }m=2 \\ P(m-1) & \text{if }m>2
					\end{cases}$$
				is a path of $\omega^y\in\SRF\lambda{\Gamma_i}$.  Then,  if $m\geq\max(3,n+1)$, $Q(m)=\ln_{m-n-1}x\in L$ and we are done.
			\end{itemize}
		\end{itemize}
		
		\item Let $i<\gamma_\lambda$ be a limit ordinal. Assume the property for $j<i$. We have that $\Gamma_i=\Unionlt ji\Gamma_j$. Again we just need to prove one inclusion, the other one being trivial. Let $x\in\Lbb$ and $y\in\SRF\lambda{\Gamma_i}$, $P\in\Pcal(y)$ and $n\in\Nbb$ minimal such that $P(n)=x$. Write $P(0)=r\omega^a$ with $a\in\Gamma_i$. Then there is $j<i$ such that $a\in\Gamma_j$. In particular $P$ is a path of $r\omega^a\in\SRF\lambda{\Gamma_j}$. We conclude using induction hypothesis on $j$.
		
	\end{itemize}
\end{proof}

\begin{Corollary}
	\label{cor:LogAtomicDeGammaUp}
	Let $\Gamma$ be an abelian additive subgroup of $\Nobf$ and
	$$L=\enstq{\exp_n x, \ln_nx}{x\in\Lbb,\quad n\in\Nbb,\quad \exists y\in\RlG\ \exists P\in\Pcal(y)\ \exists k\in\Nbb\quad P(k)=x}$$
	Then,
	$$L = \enstq{\exp_n x, \ln_nx}{\begin{array}{c}
				x\in\Lbb,\quad n\in\Nbb,\\
				\exists y\in\RlGup\ \exists P\in\Pcal(y)\ \exists k\in\Nbb\quad P(k)=x
			\end{array}}$$
\end{Corollary}

\begin{proof}
	Just apply the definition of $\RlGup$ and Lemma \ref{lem:LogAtomicDesGammai}.
	
\end{proof}

We now prove the theorem.

\begin{proof}[Proof of Theorem \ref{thm:hierarchieUparrow}]
	Let $\lambda$ be an epsilon number. Let $\mu<\mu'<\lambda$ be multiplicative ordinals. Let $x=\omega^{\omega^{-\mu}}$. Clearly, $x\in\SRF\lambda{\Nolt{\mu'}}\subseteq\SRF\lambda{{\Nolt{\mu'}}^{\uparrow\lambda}}$. So we will prove that $x\notin\SRF\lambda{{\Nolt\mu}^{\uparrow\lambda}}$.
	Note $x$ is a log-atomic number, indeed using Corollary \ref{cor:hmord} $\ln_n x=\omega^{\omega^{-\mu-n}}$. Then applying Corollary \ref{cor:LogAtomicDeGammaUp} to both $\Nolt\mu$ and $\Nolt{\mu'}$ we just need to show that 
	$$x\notin\enstq{\exp_n x, \ln_nx}{x\in\Lbb,\quad n\in\Nbb,\quad \exists y\in\SRF\lambda{\Nolt\mu}\ \exists P\in\Pcal(y)\ \exists k\in\Nbb\quad P(k)=x}$$
	Assume the converse. Then there is some path $P$ such that $P(0)\in\Rbb\omega^{\Nolt\mu}$ and there is some natural numbers $n,k\in\Nbb$ such that $P(k)=\ln_n x$.
	We prove by induction on $i$ that for all $i\in\intn0k$, $\length{a_i}\geq\mu$ where $P(i)=r_i\omega^{a_i}$, 
	\begin{itemize}
		\item For $i=k$, $P(i)=\omega^{\omega^{-\mu-n}}$ and using theorem \ref{thm:serieToSignExp}, $$\length{\omega^{-\mu-n}}=\omega\otimes(\mu+n)\geq\mu$$
		
		\item Assume the property for some $i\in\intn1k$. By definition of a path, writing $P(i-1)=r_{i-1}\omega^{a_{i-1}}$ and $a_{i-1}=\Sumlt j\nu s_j\omega^{b_j}$, there is some $j_0<\nu$ such that $b_{j_0}=g(a_i)$ and $s_{j_0}=r_i$.
		Using induction hypothesis and Corollary \ref{cor:lengthOmegaGA}, $\length{\omega^{b_{j_0}}}\geq\mu$ and therefore $\length{s_{j_0}\omega^{b_{j_0}}}\geq\mu$. Now using Lemma \ref{lem:lengthTerm}, $\length{a_{i-1}}\geq\length{s_{j_0}\omega^{b_{j_0}}}\geq\mu$.
	\end{itemize}
	The induction principle conclude that $\length{a_0}\geq\mu$. But since $P(0)\in\Rbb\omega^{\Nolt\mu}$, $\length{a_0}<\mu$. We reach a contradiction. Then $x\notin\SRF\lambda{\Nolt\mu}$.
\end{proof}

\bibliographystyle{plain}
\bibliography{bournez,perso,biblio}

\end{document}